\documentclass[a4paper]{article}
\usepackage{latexsym,bm}
\usepackage{amssymb,amsmath,amsthm}
\usepackage[english,german]{babel}
\usepackage{cite}
\usepackage[colorlinks=true]{hyperref}
\usepackage{color}

\newtheorem{theorem}{Theorem}[section]
\newtheorem{lemma}[theorem]{Lemma}
\newtheorem{remark}[theorem]{Remark}
\newtheorem{proposition}[theorem]{Proposition}
\newtheorem{definition}[theorem]{Definition}
\newtheorem{corollary}[theorem]{Corollary}
\numberwithin{equation}{section}
\hypersetup{linkcolor=blue,urlcolor=red,citecolor=red}
\title{Exponential stabilization on infinite dimensional system with impulse controls
\thanks{This work was supported by the
National Natural Science Foundation of China under grant 11701138, 11601137 and the Natural Science Foundation of Hebei Province, China under grant A2020202033. }}
\author{
Qishu Yan
\thanks{School of Science,
Hebei University of Technology, Tianjin 300400, China;  Email:
yanqishu@whu.edu.cn.}
\and
Huaiqiang Yu
\thanks{School of Mathematics, Tianjin University, Tianjin 300354, China;  Email:  huaiqiangyu@tju.edu.cn.}
\thanks{Corresponding author.}
}

\date{}
\begin{document}
\selectlanguage{english}
\maketitle

\begin{abstract}
    This paper studies the exponential stabilization on infinite dimensional system with impulse controls,
     where  impulse instants  appear periodically.
     The first main result shows that  exponential stabilizability  of the control system with a periodic feedback law is equivalent to one kind of weak observability inequalities.
     The second main result presents that, in the setting of a discrete LQ problem, the exponential stabilizability of control system with a periodic feedback law is equivalent to
    the solvability of  an algebraic Riccati-type equation which was built up  in [Qin, Wang and Yu, SIAM J. Control Optim., 59 (2021), pp. 1136-1160] for finite dimensional system.
     As an application, some sufficient and necessary condition for the exponential stabilization of  an impulse controlled system governed by coupled heat equations is given.
\end{abstract}

\section{Introduction}\label{intro}

  We start this section with some notations. Let $\mathbb{R}^+:=(0,+\infty),
\mathbb{N}:=\{0, 1, 2, \dots\}$ and $\mathbb{N}^+:=\{1, 2, \dots\}$.
    Given a Hilbert space $V$, we denote its norm and inner product by $\|\cdot\|_V$ and $\langle\cdot,\cdot\rangle_V$ respectively. Let ${\mathcal L}(V; W)$ be the space of all linear bounded operators
from a Hilbert space  $V$ to another  Hilbert space $W$ and we write $\mathcal{L}(V):=\mathcal{L}(V;V)$ for simplicity.  We denote by $\mbox{Card}E$  the number of the set of $E$. Given $\hbar\in\mathbb{N}^+$ arbitrarily.  Let
\begin{equation*}
\mathcal{J}_{\hbar}:=\{\{t_j\}_{j\in\mathbb{N}}:t_0=0<t_1<t_2<\cdots \textrm{ and }
t_{j+\hbar}-t_j=t_\hbar\;\;\mbox{for each}\;\; j\in \mathbb N\}.
\end{equation*}
Then, for each $\{t_j\}_{j\in \mathbb{N}}\in \mathcal{J}_\hbar$, we have
\begin{equation}\label{Intro-2}
t_{j+k\hbar}=t_j+t_{k\hbar}=t_j+kt_{\hbar}\;\;\mbox{for all}\;\;
j, k\in \mathbb{N}.
\end{equation}
    Given an (bounded or unbounded) operator $O$, let $O^*$ be its adjoint operator.
    Given a Hilbert space $V$, we define the following sets:
\begin{equation*}\label{1014-1}
\mathcal{SL}(V):=\{O\in \mathcal L(V): O^*=O\},\;\;\mathcal{SL}_+(V):=\{O\in \mathcal{SL}(V): \langle v,Ov\rangle_V\geq 0\,\textrm{ for any } v\in V\},
\end{equation*}
\begin{equation*}\label{yu-2-21-1}
    \mathcal{SL}_{\gg}(V):=\{O\in \mathcal{SL}_+(V):
    \textrm{ there exists }\delta>0\;\mbox{ so that }\;O-\delta I\in \mathcal{SL}_+(V)\},
\end{equation*}
\begin{equation*}\label{yu-2-21-2}
\mathcal M_{\hbar,+}(V):=\{(M_j)_{j\in\mathbb N^+}: M_j\in \mathcal{SL}_{+}(V)\;\mbox{and}\;\,M_{j+\hbar}=M_j \textrm{ for each }j\in\mathbb N^+\}
\end{equation*}
   and
\begin{equation*}\label{1014-2}
\mathcal M_{\hbar,\gg}(V):=\{(M_j)_{j\in\mathbb N^+}: M_j\in \mathcal{SL}_{\gg}(V)\;\mbox{and}\;\,M_{j+\hbar}=M_j \textrm{ for each }\,j\in\mathbb N^+\}.
\end{equation*}
\par
    Throughout this paper, we use $H$ and $U$ to denote two reflexive Hilbert spaces.  Assume that $H$ and $U$ are identified with their dual spaces respectively.
    Let the linear unbounded operator $A$ with  domain $D(A)\subset H$ be a generator of a $C_0$-semigroup $\{e^{At}\}_{t\geq 0}$ over $H$.
    Let $\mathcal{B}_{\hbar}:=\{B_k\}_{k=1}^\hbar\subset\mathcal{L}(U;H)$ and $\mathcal{B}^*_{\hbar}:=\{B_k^*\}_{k=1}^{\hbar}\subset\mathcal{L}(H;U)$.
\subsection{Control problem and aim}
    Given $\Lambda_\hbar=\{t_j\}_{j\in\mathbb{N}}\in \mathcal{J}_{\hbar}$,
we consider the following impulse controlled system:
\begin{equation}\label{Intro-3}
\begin{cases}
  \displaystyle x'(t)=Ax(t),   & t\in\mathbb R^+\backslash \Lambda_\hbar,   \\
  x(t_j^+)=x(t_j)+B_{\nu(j)} u_j,   &j\in\mathbb{N}^+,   \\
\end{cases}
\end{equation}
where $\nu(j):=j-[j/\hbar]\hbar$  for each $j\geq 1$,
$x(t_j^+)$ denotes the right limit at $t_j$ for the function $x(\cdot)$,
and $(u_j)_{j\in\mathbb{N}^+}\in l^2(\mathbb{N}^+;U)$. Here and in what follows, $[s]:=\max\{k\in\mathbb N:\,k<s\}$ for each $s>0$. We denote the system (\ref{Intro-3}) by
$[A,\mathcal{B}_{\hbar},\Lambda_{\hbar}]$ and the solution of (\ref{Intro-3}) by $x(\cdot;x_0,u)$ with control $u\in l^2(\mathbb{N}^+;U)$ and initial data $x_0\in H$.
\par

Now,  we introduce some definitions on exponential stabilization of the system $[A,\mathcal{B}_{\hbar},\Lambda_{\hbar}]$.
\begin{definition}\label{Intro-5}
    The system $[A,\mathcal{B}_{\hbar},\Lambda_{\hbar}]$
  is called  exponentially $\hbar$-stabilizable if there is a sequence of feedback laws $\{F_k\}_{k=1}^\hbar\subset \mathcal L(H;U)$
  so that the following closed-loop system is exponentially stable:
  \begin{equation}\label{Intro-4(2)}
\begin{cases}
  \displaystyle x'(t)=Ax(t),   & t\in\mathbb R^+\backslash \Lambda_\hbar,   \\
  x(t_j^+)=x(t_j)+B_{\nu(j)} F_{\nu(j)} x(t_j),   &j\in\mathbb{N}^+,   \\
\end{cases}
\end{equation}
 i.e., there are constants $C>0$ and $\mu>0$ so that any solution of \eqref{Intro-4(2)}, denoted by $x_{\mathcal F}(\cdot)$,
satisfies that
\begin{equation*}
\|x_{\mathcal F}(t)\|_{H}\leq C \mathrm e^{-\mu t}\|x_{\mathcal F}(0)\|_{H}\;\;\mbox{for any}\;\;t\in\mathbb{R}^+.
\end{equation*}
\end{definition}
   If the system $[A,\mathcal{B}_{\hbar},\Lambda_{\hbar}]$
  is  exponentially $\hbar$-stabilizable with feedback law $\{F_k\}_{k=1}^{\hbar}$, we simply write $\mathcal F:=\{F_k\}_{k=1}^\hbar$  and denote the closed-loop system \eqref{Intro-4(2)} by $[A,\{B_kF_k\}_{k=1}^\hbar, \Lambda_\hbar]$.
Since $\{F_k\}_{k=1}^\hbar$ appears at time instants $\Lambda_\hbar$ $\hbar$-periodically, the feedback law $\mathcal F$ is indeed $\hbar$-periodic time-varying.
\vskip 5pt
\emph{Aim.} In this paper, we shall mainly present some equivalence criteria for the exponential stabilization of the system $[A,\mathcal{B}_{\hbar},\Lambda_{\hbar}]$.
\subsection{Related works and motivation}
    The stabilizability of infinite dimensional system with distributed controls have been studied widely. For instance, we can refer to  \cite{wang-xu,T-W-X, L-W-X-Y,Ammari,Badra2, Bastin,Coron2, Flandoli-Lasiecka-Triggiani,Huang-Wang-Wang,
    Kunisch,Lasiecka-Triggiani,Lions, Liu, Pritchard,Russell}. Recently, the impulse control system has  attracted a lot of attention. For example, we can refer to \cite{Duan-Wang-Zhang,Duan-Wang,Q-W,Q-W-Y,Trelet-wang-zhang,P-W-X}. In particular, in \cite{Q-W-Y}, the authors studied the exponential stabilizability of finite dimensional system with impulse controls and obtained some equivalence results. Therefore, how to obtain the equivalence conditions for the exponential stabilizability of infinite dimensional system with impulse controls seems very interesting.

\subsection{Main results}
\par
    The first main result can be stated as follows.
\begin{theorem}\label{yu-main-theorem}
    Given  $[A,\mathcal{B}_{\hbar}, \Lambda_{\hbar}]$.  The following statements are equivalent:
\begin{enumerate}
  \item [(i)] The system $[A,\mathcal{B}_{\hbar}, \Lambda_\hbar]$ is exponentially $\hbar$-stabilizable.
  \item [(ii)] For each $\sigma\in (0,1)$, there are $k=k(\sigma)\in\mathbb{N}^+$ and
  $C(\sigma)>0$ so that
\begin{equation}\label{yu-3-8-1}
    \|e^{A^* t_{k}}\varphi\|_H\leq C(\sigma)
    \left(\sum_{j=1}^{k}\|B^*_{\nu(j)}e^{A^*(t_{k}-t_j)}
    \varphi\|_U^2\right)^{\frac{1}{2}}
    +\sigma \|\varphi\|_H\;\;\mbox{for any}\;\;\varphi\in H.
\end{equation}
  \item [(iii)] There are $\sigma\in(0,1)$, $k\in\mathbb{N}^+$ and $C>0$ so that
 \begin{equation}\label{0315-1}
   \|\mathrm{e}^{A^*t_{k\hbar}}\varphi\|_H \leq C\left(\sum_{j=1}^{k\hbar-1}\|B^*_{\nu(j)}\mathrm{e}^{A^*(t_{k\hbar}-t_j)}\varphi
    \|_U^2\right)^{\frac{1}{2}}
    +\sigma\|\varphi\|_H \;\;\mbox{for any}\;\;\varphi\in H.
\end{equation}
    \item [(iv)] There are $\sigma\in (0,1)$, $k\in\mathbb{N}^+$ and $C>0$ so that, for any $x_0\in H$, there is a control $u\in l^2(\mathbb{N}^+;U)$ so that
\begin{equation}\label{yu-3-9-bbb-2}
    \|x(t_{k\hbar};x_0,u)\|_H\leq \sigma\|x_0\|_H
    \;\;\mbox{and}\;\;\|u\|_{l^2(\mathbb{N}^+;U)}\leq C\|x_0\|_H.
\end{equation}
\end{enumerate}
\end{theorem}
\begin{remark}
    It is well known that, on $[0,t_k]$ $(k\in\mathbb{N}^+)$, the dual-observe system of $[A,\mathcal{B}_{\hbar}, \Lambda_\hbar]$ is given by
\begin{equation*}\label{yu-3-10-1}
\begin{cases}
    y'(t)+A^*y(t)=0,&t\in[0,t_k],\\
    y(t_k)=\varphi\in H,\\
    z(t_j)=B^*_{\nu(j)}y(t_j),&j\in\{1,2,\ldots,k\}.
\end{cases}
\end{equation*}
    We denote this observe system by $[A^*,\mathcal{B}^*_\hbar,\Lambda_{\hbar}]$.
    From the viewpoint of observability, the inequality (\ref{yu-3-8-1}) is called a weak observability of  $[A^*,\mathcal{B}^*_\hbar,\Lambda_{\hbar}]$ on $[0,t_k]$.
    It should be noted that, in
    infinite dimensional setting, the strong observability inequality (i.e., $\sigma=0$ in (\ref{yu-3-8-1})),
    which means the corresponding control system is null controllable, does not hold generally (see, for instance, \cite{Q-W}).
\end{remark}
    The proof of Theorem \ref{yu-main-theorem} is based on the study of a discrete LQ problem associated with  the system $[A,\mathcal{B}_{\hbar}, \Lambda_\hbar]$.
    We now introduce this problem and present another equivalence criterion for the exponential stabilization of the system $[A,\mathcal{B}_{\hbar},\Lambda_{\hbar}]$.
\par
   We arbitrarily fix $\Lambda_\hbar=\{t_j\}_{j\in \mathbb N}\in \mathcal J_\hbar$,
   $(Q_j)_{j\in\mathbb N^+}\in \mathcal M_{\hbar,\gg}(H)$ and $(R_j)_{j\in\mathbb N^+}\in \mathcal M_{\hbar,\gg}(U)$.
    For each  $x_0\in H$, define the admissible control set:
\begin{equation}\label{1015-4}
\mathcal U_{ad}(x_0):=\{u=(u_j)_{j\in\mathbb N^+}\in l^2(\mathbb N^+;U):\, \left(x(t_j;x_0,u)\right)_{j\in\mathbb N^+}\in l^2(\mathbb N^+;H)\}
\end{equation}
and the cost functional:
\begin{equation}\label{1015-5-1}
\begin{split}
J(u;x_0):=\sum_{j=1}^{+\infty}\left(\langle Q_jx(t_j;x_0,u),x(t_j;x_0,u)\rangle_{H}+\langle R_ju_j,u_j\rangle_{U}\right)
\textrm{ for any }u=(u_j)_{j\in\mathbb N^+}\in \mathcal U_{ad}(x_0).
\end{split}
\end{equation}
    Consider the following optimal control problem:
\vskip 5pt
{\textbf{(I-I-LQ)}}: Given $x_0\in H$.   Find a control $u^*=(u_j^*)_{j\in\mathbb N^+}\in \mathcal U_{ad}(x_0)$ so that
\begin{equation*}\label{1015-3}
J(u^*;x_0)=\inf_{u\in\mathcal U_{ad}(x_0)}J(u;x_0).
\end{equation*}
\par
     The problem {\textbf{(I-I-LQ)}} is a kind of discrete LQ problems associated with  the system $[A,\mathcal{B}_{\hbar}, \Lambda_\hbar]$.
     For solving it, we first introduce the following operator-valued algebraic equation:
\begin{equation}\label{1014-3}
\begin{cases}
P_{k-1}-\mathrm e^{A^*(t_{k}-t_{k-1})}P_{k}\mathrm e^{A(t_{k}-t_{k-1})}
=\mathrm e^{A^*(t_{k}-t_{k-1})}Q_{k}\mathrm e^{A(t_{k}-t_{k-1})}\\
\qquad\qquad -\mathrm e^{A^*(t_{k}-t_{k-1})}P_{k}B_{k}(R_{k}+B_{k}^* P_{k}B_{k})^{-1} B^*_{k}P_{k}\mathrm e^{A(t_{k}-t_{k-1})},& k\in\{1,2,\ldots,\hbar\},\\
P_0=P_\hbar.
\end{cases}
\end{equation}
    We call  \eqref{1014-3} a Riccati-type equation, in which  unknowns $P_k$ ($k=1,\ldots, \hbar$) are in $\mathcal{SL}_{+}(H)$. The solution to \eqref{1014-3}, if exists, is denoted by $\{P_k\}_{k=0}^\hbar$.
\par
    The second main result of this paper is given as follows.
\begin{theorem}\label{thm-1}
Given $[A,\mathcal{B}_\hbar, \Lambda_\hbar]$. The following statements are equivalent:
\begin{enumerate}
\item[(i)] The system $[A,\mathcal{B}_\hbar, \Lambda_\hbar]$ is exponentially $\hbar$-stabilizable.

\item[(ii)] For each $x_0\in H$, the  set $\mathcal U_{ad}(x_0)$ is not empty.

\item[(iii)] For any $(Q_j)_{j\in\mathbb N^+}\in\mathcal M_{\hbar,\gg}(H)$ and $(R_j)_{j\in\mathbb N^+}\in\mathcal M_{\hbar,\gg}(U)$, the
Riccati-type equation \eqref{1014-3} has a unique solution.

\item[(iv)] There are $(Q_j)_{j\in\mathbb N^+}\in\mathcal M_{\hbar,\gg}(H)$ and $(R_j)_{j\in\mathbb N^+}\in\mathcal M_{\hbar,\gg}(H)$ so that the
Riccati-type equation \eqref{1014-3} has a unique solution.
\end{enumerate}
    Moreover, when the system \eqref{1014-3} is exponentially $\hbar$-stabilizable, the feedback law  $\mathcal F=\{F_k\}_{k=1}^\hbar$
can be chosen as
\begin{equation*}\label{1014-6}
 F_k:=-(R_k+B_k^*P_k B_k)^{-1} B_k^*P_k \textrm{ for each }k=1,...,\hbar,
 \end{equation*}
where $\{P_k\}_{k=0}^\hbar$ is the solution of \eqref{1014-3} with arbitrarily fixed $(Q_j)_{j\in\mathbb N^+}\in \mathcal M_{\hbar, \gg}(H)$ and
 $(R_j)_{j\in\mathbb N^+}\in \mathcal M_{\hbar, \gg}(U)$.
\end{theorem}
\subsection{Novelties of this paper}
\begin{itemize}
  \item As an important control system, there is very little literature  considering
    the stabilization problem with impulse controls in infinite dimensional setting.
    In essence, the controllers in the impulse control system are time-varying. Thus, the classic strategies (for instance, frequency domain method) do not work.
     Indeed, Theorem \ref{yu-main-theorem} can be regarded as a time domain method (for the time domain method in continous control system, one can refer to \cite{T-W-X,L-W-X-Y}).
  \item It seems for us that the equivalence characterization on the exponential stabilization for impulse control systems in infinite dimensional setting has not been touched upon (at least, we do not find any such literature). Thus, the equivalence results in Theorem \ref{yu-main-theorem} and Theorem \ref{thm-1} are new.
  \item In  infinite dimensional setting, the discrete LQ problem used in our paper seems new for us. Although many literatures used LQ problem to consider the stabilizability of infinite dimensional systems (for instance, \cite{Flandoli-Lasiecka-Triggiani,Lasiecka-Triggiani}), but our LQ problem is discrete. It is  completely different from continuous time horizon LQ problems.
\end{itemize}

\subsection{Plan of this paper}
    The rest of this paper is organized as follows: In section 2, we present some preliminaries. In section 3, we give the proof of our main results. In Section 4, we apply our main results to the system of heat equations coupled by constant matrices with impulse controls.
\section{Preliminaries}

\subsection{Finite horizon LQ problem}\label{yu-sec-2.1}
     Given $\Lambda_{\hbar}=\{t_j\}_{j\in\mathbb{N}}\in\mathcal{J}_{\hbar}$. We arbitrarily fix $\mathcal (Q_j)_{j\in\mathbb N^+}\in\mathcal M_{\hbar, +}(H)$ and
 $(R_j)_{j\in\mathbb N^+}\in \mathcal M_{\hbar, \gg}(U)$  and $M\in \mathcal{SL}_{+}(H)$.
We also fix a terminal instant $t_{\widehat{k}}$ with $\widehat{k}\in\mathbb N^+$.
For each $\ell\in\mathbb{N}$ and any $x_0\in H$,  we denote by  $x(\cdot;x_0,v,\ell)$  the solution of the following system:
\begin{equation}\label{1014-8}
\begin{cases}
  \displaystyle x'(t)=Ax(t),   & t\in\mathbb R^+\backslash \Lambda_\hbar,   \\
   x(t_j^+)=x(t_j)+B_{\nu(j)} v_j, & j>\ell,\\
  x(t_\ell^+)=x_0,
\end{cases}
\end{equation}
where  $v=(v_{j+\ell})_{j\in\mathbb{N}^+}\in l^2(\mathbb{N}^+;U)$ is the control. Here and in what follows, for a given Hilbert space $V$, $(f_{j+\ell})_{j\in\mathbb{N}^+}\in l^2(\mathbb{N}^+;V)$ means that, if we let
$g_j:=f_{j+\ell}$ for each $j\in\mathbb{N}^+$, then $(g_j)_{j\in\mathbb{N}^+}\in l^2(\mathbb{N}^+;V)$.
\par
 Given $x_0\in H$. For each $\ell\in\mathbb{N}$ with
$\ell<\widehat{k}$, we define
\begin{equation}\label{1014-7-1}
\begin{split}
J(v;x_0,\ell,\widehat{k}):=&\sum_{j=\ell+1}^{\widehat{k}}\left(\langle Q_j x(t_j;x_0,v,\ell), x(t_j;x_0,v,\ell)\rangle_H+\langle R_j v_j,v_j\rangle_U\right)\\
&\quad+\langle Mx(t_{\widehat{k}}^+;x_0,v,\ell),x(t_{\widehat{k}}^+;x_0,v,\ell)\rangle_H\textrm{ for each } v=(v_{j+\ell})_{j\in \mathbb{N}^+}\in l^2(\mathbb{N}^+;U).
\end{split}
\end{equation}
It should be noted that, in the definition of $J(v;x_0,\ell,\widehat{k})$,  the effective
    controls are $\{v_{j}\}_{j=\ell+1}^{\widehat{k}}$.
\par
   Now, we consider the following finite horizon LQ problem associated with  the system $[A,\mathcal{B}_{\hbar}, \Lambda_\hbar]$:
\vskip 5pt
{\textbf{(F-I-LQ)$_{\ell,\widehat{k}}$}}: Given $x_0\in H$, find a control $v(\ell,\widehat{k})=(v_{j+\ell}(\ell,\widehat{k}))_{j\in\mathbb{N}^+}\in l^2(\mathbb{N}^+;U)$ so that
\begin{equation*}\label{1014-7}
 J(v(\ell,\widehat{k});x_0,\ell,\widehat{k})=\inf_{v=(v_{j+\ell})\in l^2(\mathbb{N}^+;U)} J(v;x_0,\ell,\widehat{k}).
\end{equation*}
    Since $(R_j)_{j\in\mathbb N^+}\in \mathcal M_{\hbar, \gg}(U)$, by some standard arguments, one can easily show that
\begin{enumerate}
  \item [(a1)] {\textbf{(F-I-LQ)$_{\ell,\widehat{k}}$}} has an optimal control for any $x_0\in H$;
  \item [(a2)] $(v_{\ell+1}(\ell,\widehat{k}),\cdots,v_{\widehat{k}}(\ell,\widehat{k}))$ (called the restriction of $v(\ell,\widehat{k})$ on $[\ell+1,\widehat{k}]$) is unique
      for any $x_0\in H$.
\end{enumerate}

\par Fix $\ell\in\mathbb{N}$ with $\ell<\widehat k$. Let
\begin{equation}\label{1014-9}
V(x_0;\ell,\widehat{k}):=\inf_{v=(v_{j+\ell})\in l^2(\mathbb{N}^+;U)} J(u;x_0,\ell,\widehat{k})\;\;\mbox{for any}\;\;x_0\in H
\end{equation}
 and $\{P_j^{\widehat{k}}\}_{j=1}^{\widehat{k}}$ be the solution of the following Riccati-type equation:
\begin{equation}\label{1014-10}
\begin{cases}
P^{\widehat{k}}_{j-1}-\mathrm e^{A^*(t_{j}-t_{j-1})}P^{\widehat{k}}_{j}\mathrm e^{A(t_{j}-t_{j-1})}
=\mathrm e^{A^*(t_{j}-t_{j-1})}Q_{j}\mathrm e^{A(t_{j}-t_{j-1})}\\
\qquad\qquad -\mathrm e^{A^*(t_{j}-t_{j-1})}P^{\widehat{k}}_{j}B_{\nu(j)}(R_{j}+B_{\nu(j)}^* P^{\widehat{k}}_{j}B_{\nu(j)})^{-1} B_{\nu(j)}^*P^{\widehat{k}}_{j}\mathrm e^{A(t_{j}-t_{j-1})},&j\in\{1,2,\ldots,\widehat{k}\},\\
P^{\widehat{k}}_{\widehat{k}}=M.
\end{cases}
\end{equation}
\begin{remark}\label{yu-reamrk-4-7-1}
    Since $\mathcal (Q_j)_{j\in\mathbb N^+}\in\mathcal M_{\hbar, +}(H)$ and
 $(R_j)_{j\in\mathbb N^+}\in \mathcal M_{\hbar, \gg}(U)$  and $M\in\mathcal{SL}_+(H)$, it is clear that, for any $x_0\in H$,
$$
    J(v;x_0,\ell,\widehat{k})\geq 0\;\;\mbox{for all}\;\;v\in l^2(\mathbb{N}^+;U).
$$
    This implies that $V(x_0;\ell,\widehat{k})\geq 0$ for all $x_0\in H$. Moreover, since $M\in\mathcal{SL}_+(H)$, we can conclude that the equation (\ref{1014-10}) has a unique solution $\{P_j^{\widehat{k}}\}_{j=1}^{\widehat{k}}$ satisfying $P_j^{\widehat{k}}=
    (P_j^{\widehat{k}})^*$ for each $j=1,2,\ldots, \widehat{k}$.
\end{remark}

\par
     Next, we state the relationship between the optimal control problem {\textbf{(F-I-LQ)$_{\ell,\widehat{k}}$}}
and the solution of the Riccati-type equation \eqref{1014-10}.
\begin{proposition}\label{thm-filq}
For each $\ell\in\mathbb N$ with $\ell< \widehat{k}$, it stands that
\begin{equation}\label{1014-11}
V(x_0;\ell,\widehat{k})=\langle P_\ell^{\widehat{k}}x_0,x_0\rangle_H\;\;\mbox{for any}\;\;x_0\in H.
\end{equation}
%{\color{red}Moreover, for any $x_0\in H$, if the optimal control of {\textbf{(F-I-LQ)}}$_{\ell,\widehat{k}}$ with respect to the
%initial state $x_0$ is $v(\ell,\widehat{k})=(v_{j+\ell}(\ell,\widehat{k}))_{j\in\mathbb{N}^+}$, then
%\begin{equation}\label{1014-17}
%v_{j}(\ell,\widehat{k})=-(R_j+B_j^*P^{\widehat{k}}_jB_j)^{-1}B_j^*
%P^{\widehat{k}}_jx(t_j;x_0,v(\ell,\widehat{k}),\ell),
%\mbox{for each}\,j\in\{\ell+1,\ell+2,\ldots, \widehat{k}\}.
%\end{equation}}
\end{proposition}
\begin{proof}
Let $x_0\in H$ be arbitrarily fixed.
The proof is given by induction.
\vskip 5pt
{\it Step 1. We show that \eqref{1014-11} stands for $\ell=\widehat{k}-1$.}
\par
In fact, by \eqref{1014-7-1}, one can easily check that, for any $v=(v_{j+\widehat{k}-1})_{j\in\mathbb{N}^+}\in l^2(\mathbb{N}^+;U)$,
\begin{equation}\label{1014-12}
\begin{split}
J(v;x_0,\ell,\widehat{k})=&J(v;x_0,\widehat{k}-1,\widehat{k})\\
=&\langle Q_{\widehat{k}} x(t_{\widehat{k}};x_0,v,\widehat{k}-1), x(t_{\widehat{k}};x_0,v,\widehat{k}-1)\rangle_H+\langle R_{\widehat{k}} v_{\widehat{k}},v_{\widehat{k}}\rangle_U\\
&+\langle Mx(t_{\widehat{k}}^+;x_0,v,\widehat{k}-1),x(t_{\widehat{k}}^+;x_0,v,\widehat{k}-1)\rangle_H.
\end{split}
\end{equation}
   By (\ref{1014-8}), it is obvious that
\begin{equation}\label{1014-13}
x(t_{\widehat{k}};x_0,v,\widehat{k}-1)=\mathrm e^{A(t_{\widehat{k}}-t_{\widehat{k}-1})}x_0 \textrm{ and }x(t_{\widehat{k}}^+;x_0,u,\widehat{k}-1)=\mathrm e^{A(t_{\widehat{k}}-t_{\widehat{k}-1})}x_0+B_{\nu(\widehat{k})} u_{\widehat{k}}.
\end{equation}
    Thus, by (\ref{1014-12}) and (\ref{1014-13}), we have
\begin{eqnarray}\label{1014-14}
J(v;x_0,\widehat{k}-1,\widehat{k})&=&\langle(R_{\widehat{k}}+B_{\nu(\widehat{k})}^*
MB_{\nu(\widehat{k})})v_{\widehat{k}},v_{\widehat{k}}\rangle_U+2 \langle v_{\widehat{k}},B_{\nu(\widehat{k})}^*M\mathrm e^{A(t_{\widehat{k}}-t_{{\widehat{k}}-1})}x_0 \rangle_U\nonumber\\
&&+\langle (Q_{\widehat{k}}+M) \mathrm e^{A(t_{\widehat{k}}-t_{{\widehat{k}}-1})}x_0,\mathrm e^{A(t_{\widehat{k}}-t_{{\widehat{k}}-1})}x_0 \rangle_H\nonumber\\
&=&\|(R_{\widehat{k}}+B_{\nu(\widehat{k})}^*MB_{\nu(\widehat{k})})^{\frac{1}{2}}
(v_{\widehat{k}}+(R_{\widehat{k}}+B_{\nu(\widehat{k})}^*M
B_{\nu(\widehat{k})})^{-1}B_{\nu(\widehat{k})}^*M\mathrm e^{A(t_{\widehat{k}}-t_{\widehat{k}-1})}x_0)\|_U^2\nonumber\\
&&-\langle( R_{\widehat{k}}+B_{\nu(\widehat{k})}^*MB_{\nu(\widehat{k})} )^{-1}B_{\nu(\widehat{k})} ^*M\mathrm e^{A(t_{\widehat{k}}-t_{\widehat{k}-1})}x_0,B_{\nu(\widehat{k})}^*M\mathrm e^{A(t_{\widehat{k}}-t_{\widehat{k}-1})}x_0\rangle_U\nonumber\\
&&+\langle (Q_{\widehat{k}}+M) \mathrm e^{A(t_{\widehat{k}}-t_{\widehat{k}-1})}x_0,\mathrm e^{A(t_{\widehat{k}}-t_{\widehat{k}-1})}x_0 \rangle_H.
\end{eqnarray}
   It follows that
\begin{eqnarray}\label{yu-2-25-1}
    &\;&J(v;x_0,\widehat{k}-1,\widehat{k})\nonumber\\
    &\geq&-\langle(R_{\widehat{k}}+B_{\nu(\widehat{k})} ^*
    MB_{\nu(\widehat{k})})^{-1}B_{\nu(\widehat{k})} ^*M\mathrm e^{A(t_{\widehat{k}}-t_{\widehat{k}-1})}x_0,B_{\nu(\widehat{k})} ^*M\mathrm e^{A(t_{\widehat{k}}-t_{\widehat{k}-1})}x_0\rangle_U\nonumber\\
&\;&+\langle (Q_{\widehat{k}}+M) \mathrm e^{A(t_{\widehat{k}}-t_{\widehat{k}-1})}x_0,\mathrm e^{A(t_{\widehat{k}}-t_{\widehat{k}-1})}x_0 \rangle_H\;\;\mbox{for any}\;\;v=(v_{j+\widehat{k}-1})_{j\in\mathbb{N}^+}\in l^2(\mathbb{N}^+;U).
\end{eqnarray}
    Taking $v^*=(v_{j+\widehat{k}-1})_{j\in\mathbb{N}^+}\in l^2(\mathbb{N}^+; U)$ with
\begin{eqnarray*}\label{yu-2-25-3}
    v_{\widehat{k}}&=&-(R_{\widehat{k}}+B_{\nu(\widehat{k})} ^*MB_{\nu(\widehat{k})} )^{-1}
    B_{\nu(\widehat{k})} ^*M\mathrm e^{A(t_{\widehat{k}}-t_{\widehat{k}-1})}x_0\nonumber\\
    &=&-(R_{\widehat{k}}+B_{\nu(\widehat{k})} ^*MB_{\nu(\widehat{k})} )^{-1}
    B_{\nu(\widehat{k})} ^*P_{\widehat{k}}^{\widehat{k}}x(t_{\widehat{k}};x_0,v,\widehat{k}-1)
\end{eqnarray*}
     into (\ref{1014-14}), by (\ref{yu-2-25-1}), we get
\begin{eqnarray}\label{yu-2-25-2}
    J(v^*;x_0,\widehat{k}-1,\widehat{k})&=&-\langle(R_{\widehat{k}}
    +B_{\nu(\widehat{k})} ^*MB_{\nu(\widehat{k})} )^{-1}B_{\nu(\widehat{k})} ^*M\mathrm e^{A(t_{\widehat{k}}-t_{\widehat{k}-1})}x_0,B_{\nu(\widehat{k})} ^*M\mathrm e^{A(t_{\widehat{k}}-t_{\widehat{k}-1})}x_0\rangle_U\nonumber\\
&\;&+\langle (Q_{\widehat{k}}+M) \mathrm e^{A(t_{\widehat{k}}-t_{\widehat{k}-1})}x_0,\mathrm e^{A(t_{\widehat{k}}-t_{\widehat{k}-1})}x_0 \rangle_H.
\end{eqnarray}
    By (\ref{yu-2-25-1}) and (\ref{yu-2-25-2}), we can conclude that
\begin{eqnarray*}\label{1014-16}
V(x_0;\widehat{k}-1,\widehat{k})&=&J(v^*;x_0,\widehat{k}-1,\widehat{k})=\langle (Q_{\widehat{k}}+M) \mathrm e^{A(t_{\widehat{k}}-t_{\widehat{k}-1})}x_0,\mathrm e^{A(t_{\widehat{k}}-t_{\widehat{k}-1})}x_0 \rangle_H\nonumber\\
&\;&-\langle(R_{\widehat{k}}+B_{\nu(\widehat{k})} ^*MB_{\nu(\widehat{k})} )^{-1}B_{\nu(\widehat{k})}^*M\mathrm e^{A(t_{\widehat{k}}-t_{\widehat{k}-1})}x_0,
B_{\nu(\widehat{k})}^*M\mathrm e^{A(t_{\widehat{k}}-t_{\widehat{k}-1})}x_0\rangle_U.
\end{eqnarray*}
This, together with (\ref{1014-10}), implies that \eqref{1014-11} stands for $\ell=\widehat{k}-1$.
\vskip 5pt
{\it Step 2. Let $\ell'\in\mathbb N$ with $\ell'<\widehat{k}-1$ be fixed arbitrarily.  Suppose that for $\ell=\ell'+1,\ell'+2,\ldots,\widehat{k}-1$, \eqref{1014-11} stands.
We prove that \eqref{1014-11} holds for $\ell=\ell'$.}

In fact, for each $v=(v_{j+\ell'})_{j\in\mathbb{N}^+}\in l^2(\mathbb{N}^+;U)$, we have
\begin{eqnarray*}\label{1014-18}
J(v;x_0,\ell',\widehat{k})&=&\sum_{j=\ell'+1}^{\widehat{k}}\left(\langle Q_j x(t_j;x_0,v,\ell'), x(t_j;x_0,v,\ell')\rangle_H+\langle R_j v_j,v_j\rangle_U\right)\nonumber\\
&&\quad+\langle Mx(t_{\widehat{k}}^+;x_0,v,\ell'),x(t_{\widehat{k}}^+;x_0,v,\ell')\rangle_H\nonumber\\
&=& \left\langle Q_{\ell'+1} x(t_{\ell'+1};x_0,v,\ell'), x(t_{\ell'+1};x_0,v,\ell')\right\rangle_H+\left\langle R_{\ell'+1}v_{\ell'+1}, v_{\ell'+1}\right\rangle_U\\
&&\quad +J\left(v';x(t^+_{\ell'+1};v,x_0,\ell'),\ell'+1,\widehat{k}\right),\nonumber
\end{eqnarray*}
where $v':=(v_{j+\ell'+1})_{j\in\mathbb{N}^+}\in l^2(\mathbb{N}^+;U)$.
This, along with the assumption that  \eqref{1014-11}  stands for $\ell=\ell'+1$, shows that, for any $v=(v_{j+\ell'})_{j\in\mathbb N^+}\in l^2(\mathbb{N}^+;U)$,
\begin{eqnarray*}
&&J(v;x_0,\ell',\widehat{k})\\
&\geq& \langle Q_{\ell'+1} x(t_{\ell'+1};x_0,v,\ell'), x(t_{\ell'+1};x_0,v,\ell')\rangle_H+\langle R_{\ell'+1}v_{\ell'+1}, v_{\ell'+1}\rangle_U\\
&&+\langle P^{\widehat{k}}_{\ell'+1}x(t^+_{\ell'+1};v,x_0,\ell'), x(t^+_{\ell'+1};v,x_0,\ell')\rangle_H.
\end{eqnarray*}
Then by the same arguments as those in \emph{Step 1}, we obtain that
\begin{equation}\label{1014-19}
J(v;x_0,\ell',\widehat{k})\geq \langle P_{\ell'}^{\widehat{k}} x_0,x_0\rangle \;\;\mbox{for any}\;\;
v=(v_{j+\ell'})_{j\in\mathbb{N}^+}\in l^2(\mathbb{N}^+;U),
\end{equation}
and when we take $v=(v_{j+\ell'})_{j\in\mathbb{N}^+}\in l^2(\mathbb{N}^+;U)$ with
\begin{equation*}\label{1014-20}
v_j=-(R_j+B_jP^{\widehat{k}}_jB_j)^{-1}B_j^*P^{\widehat{k}}_jx(t_j;v,x_0,\ell') \;\;\mbox{for each}\;\; j=\ell'+1,...,\widehat{k}
\end{equation*}
into (\ref{1014-19}), the inequality (\ref{1014-19}) achieves  an equality.
\par
    We complete the proof.
\end{proof}
    By Remark \ref{yu-reamrk-4-7-1} and Proposition \ref{thm-filq}, we have the following corollary:
\begin{corollary}\label{yu-corollary-4-7-2}
    The solution $\{P_j^{\widehat{k}}\}_{j=1}^{\widehat{k}}$ of the Riccati-type equation
     (\ref{1014-10}) belongs to $\mathcal{SL}_{+}(H)$.
\end{corollary}

\subsection{Infinite horizon LQ problem}

    Arbitrarily fix $\Lambda_\hbar=\{t_j\}_{j\in \mathbb N}\in \mathcal J_\hbar$,  $(Q_j)_{j\in\mathbb N^+}\in \mathcal M_{\hbar,\gg}(H)$ and $(R_j)_{j\in\mathbb N^+}\in \mathcal M_{\hbar,\gg}(U)$.
    Let $\mathcal{U}_{ad}(\cdot)$ and $J(\cdot;\cdot)$ be defined by (\ref{1015-4}) and
    (\ref{1015-5-1}), respectively. Throughout this subsection, we always suppose that
\begin{equation}\label{yu-3-0-1}
    \mathcal{U}_{ad}(x_0)\neq \emptyset\;\;\mbox{for any}\;\;x_0\in H.
\end{equation}
\par
    Now, we consider the problem
{\textbf{(I-I-LQ)}}.
    From (\ref{yu-3-0-1}), the following two facts should be noted:
\begin{enumerate}
  \item [(b1)] By a standard argument, one can show that the problem {\textbf{(I-I-LQ)}} has a unique optimal control for any $x_0\in H$;
  \item [(b2)] For any $x_0\in H$,  $\inf_{u\in\mathcal U_{ad}(x_0)}J(u;x_0)=\inf_{u\in l^2(\mathbb{N}^+;U)}J(u;x_0)<+\infty$. Here and in what follows, we permit $J(u;x_0)=+\infty$ if $u\in l^2(\mathbb{N}^+;U)\setminus \mathcal{U}_{ad}(x_0)$.
\end{enumerate}
\par
    For any $x_0\in H$, we define the value function of the problem {(\textbf{I-I-LQ})} associated with $x_0\in H$ as follows:
\begin{equation*}\label{1015-5}
V(x_0):=\inf_{u\in\mathcal U_{ad}(x_0)} J(u;x_0).
\end{equation*}
   Since (b2) stands, it is clear that $V(x_0)<+\infty$ for any $x_0\in H$.
   For each $\ell\in\mathbb N$ and any $x_0\in H$, we define
\begin{equation*}\label{1015-6}
\mathcal U_{ad}(x_0;\ell):=\{u=(u_{j+\ell})_{j\in\mathbb{N}^+}\in l^2(\mathbb N^+;U): (x(t_{j+\ell};u,x_0,\ell))_{j\in\mathbb{N}^+}\in l^2(\mathbb N^+;H)\},
\end{equation*}
and
\begin{equation}\label{yu-b-3-4-1}
\begin{split}
J(u;x_0,\ell):=\sum_{j=\ell+1}^{+\infty}\left(\langle Q_jx(t_j;x_0,u,\ell),x(t_j;x_0,u,\ell)\rangle_H+\langle R_ju_j,u_j\rangle_U \right)\qquad\qquad\\
 \textrm{ for any }u=(u_{j+\ell})_{j\in\mathbb{N}^+}\in \mathcal U_{ad}(x_0;\ell),
\end{split}
\end{equation}
    where $x(\cdot;x_0,u,\ell)$ is the solution of the equation (\ref{1014-8}) with
    initial data $x_0$ and control $u=(u_{j+\ell})_{j\in \mathbb{N}^+}\in l^2(\mathbb{N}^+;U)$.
One can easily check that, for any $x_0\in H$,
\begin{equation}\label{1015-7}
\mathcal U_{ad}(x_0;0)=\mathcal U_{ad}(x_0),
\end{equation}
and
\begin{equation*}\label{1015-8}
J(u;x_0,0)=J(u;x_0)\;\;\mbox{for any}\;\;u\in \mathcal U_{ad}(x_0).
\end{equation*}
\par

\begin{lemma}\label{lem-3-1}
For each $\ell\in\mathbb N$ and $x_0\in H$, $\mathcal U_{ad}(x_0;\ell)\neq \emptyset$.
\end{lemma}
\begin{proof}
Arbitrarily fix $\ell\in \mathbb N$ and $x_0\in H$.
First of all, by (\ref{yu-3-0-1}) and \eqref{1015-7}, we have
\begin{equation*}\label{1015-17}
\mathcal U_{ad}(x_0;0)\neq \emptyset,
\end{equation*}
 i.e., there exists a control $\widetilde u=(\widetilde u_j)_{j\in\mathbb{N}^+}\in l^2(\mathbb N^+;U)$
so that
\begin{equation}\label{1015-13}
(x(t_j;x_0, \widetilde u))_{j\in\mathbb{N}^+}\in l^2(\mathbb N^+;H).
\end{equation}
We now claim that
\begin{equation}\label{1015-10}
\mathcal U_{ad}(x_0;N\hbar)\neq \emptyset \textrm{ for any }N\in\mathbb N^+.
\end{equation}
Actually, let $\widetilde u^N=(\widetilde u^N_{j+N\hbar})_{j\in\mathbb{N}^+}$ be defined as
%\begin{equation}\label{1015-11}
% \widetilde u^N_j :=
% \begin{cases}
%  \widetilde u_{j-N\hbar}, & \forall\,j> N\hbar,\\
%  0, & \forall\,1\leq j\leq N\hbar.
%  \end{cases}
%\end{equation}
\begin{equation*}\label{1015-11}
 \widetilde u^N_j := \widetilde u_{j-N\hbar}\;\;\mbox{for each}\;\;j> N\hbar.
\end{equation*}
    It is clear that $(\tilde{u}_{j+N\hbar})_{j\in\mathbb{N}^+}\in l^2(\mathbb{N}^+;U)$.
    Moreover, we can directly check that
\begin{equation}\label{1015-12}
x(t_{j+N\hbar};x_0,\widetilde u^N,N\hbar)=x(t_j;x_0,\widetilde{u},0)\textrm{ for all }j\in\mathbb N^+.
\end{equation}
Thus,  by \eqref{1015-13} and \eqref{1015-12}, we have $\widetilde u^N\in \mathcal U_{ad}(x_0;N\hbar)$. It follows that (\ref{1015-10}) is true.
\par
    Let $N:=[\ell/\hbar]$. It is obvious that $N\hbar<\ell\leq (N+1)\hbar$. By the arbitrariness of $x_0$ and \eqref{1015-10}, there exists a control
\begin{equation}\label{1015-14}
\widehat v=(\widehat v_{j+(N+1)\hbar})_{j\in\mathbb{N}^+}\in\mathcal U_{ad}(\mathrm e^{A(t_{(N+1)\hbar}-t_\ell)}x_0;(N+1)\hbar),
\end{equation}
    i.e.,
\begin{equation}\label{yu-3-3-4}
    (x(t_{j+(N+1)\hbar};\mathrm e^{A(t_{(N+1)\hbar}-t_{\ell})}x_0,
    \widehat{v},(N+1)\hbar))_{j\in\mathbb{N}^+}\in l^2(\mathbb{N}^+;H).
\end{equation}
Define $\widehat u=(\widehat u_{j+\ell})_{j\in\mathbb N^+}$ in the manner:
\begin{equation*}\label{1015-15}
 \widehat u_j:=\begin{cases}
 0 & \textrm{ when }\ell< j\leq(N+1)\hbar,\\
 \widehat v_j &\textrm{ when }j>(N+1)\hbar.
 \end{cases}
\end{equation*}
    By (\ref{1015-14}), it is clear that $\widehat{u}\in l^2(\mathbb{N}^+;U)$. Moreover,
   one can easily check that
\begin{equation*}\label{1015-16}
x(t_j;\widehat u,x_0,\ell)=\begin{cases}
\mathrm e^{A(t_j-t_\ell)}x_0, & \textrm{ if }\ell\leq j\leq (N+1)\hbar,\\
x(t_j;\widehat v,\mathrm e^{A(t_{(N+1)\hbar}-t_\ell)}x_0,(N+1)\hbar), &\textrm{ if }j>(N+1)\hbar.
\end{cases}
\end{equation*}
This, along with \eqref{yu-3-3-4}, yields that $(x(t_{j+\ell};\widehat u,x_0,\ell))_{j\in\mathbb{N}^+}\in l^2(\mathbb N^+;H)$. Hence, $\widehat u\in \mathcal U_{ad}(x_0;\ell)$, i.e., $\mathcal U_{ad}(x_0;\ell)\neq \emptyset$.
This ends the proof.
\end{proof}
     Now, for each $\ell\in\mathbb N$, we study the following  LQ problem:
\vskip 5pt
{(\textbf{I-I-LQ})}$_{\ell}$:  Given $x_0\in H$. Find a control $u^\ell\in\mathcal U_{ad}(x_0;\ell)$ so that
\begin{equation}\label{1015-9}
V(x_0;\ell):=\inf_{u\in\mathcal U_{ad}(x_0;\ell)} J(u;x_0,\ell)=J(u^\ell;x_0,\ell).
\end{equation}
We call $V(\cdot;\ell)$ the value
function of {\textbf{(I-I-LQ)}$_\ell$}.
It is clear that ({\textbf{I-I-LQ}})$_{0}$ coincides with ({\textbf{I-I-LQ}}) and $V(\cdot)=V(\cdot;0)$.
\begin{lemma}\label{lem-3-2}
     The following two statements are true:
\begin{enumerate}
\item[(i)] For each $\ell\in\mathbb N$, there is a unique $P_\ell\in \mathcal{SL}_{+}(H)$, so that
\begin{equation}\label{1015-2}
V(x_0;\ell)=\langle P_\ell x_0,x_0\rangle_H \textrm{ for any }x_0\in H.
\end{equation}
\item[(ii)]  $P_{\ell+\hbar}=P_\ell$ for all $\ell \in \mathbb N$.
\end{enumerate}
\end{lemma}

\begin{proof}
Let $\ell\in\mathbb N$ be arbitrarily fixed.
    We first present the proof of $(i)$.  The proof is divided into 3 steps.

\vskip 5pt

    \emph{Step 1. We prove that there exists a constant  $K>0$ so that
\begin{equation}\label{1016-1}
   V(x_0;\ell)\leq K\|x_0\|_H^2\;\;\mbox{for any}\;\;x_0\in H.
\end{equation}}

 Let $x_0\in H$ be arbitrarily fixed.  By  Lemma \ref{lem-3-1}, we know that $\mathcal U_{ad}(x_0;\ell)\neq \emptyset$ which implies that  $0\leq V(x_0;\ell)<+\infty$ obviously.
Thus the problem (\textbf{I-I-LQ})$_\ell$
has a unique optimal control $u^{\ell}=(u^\ell_{j+\ell})_{j\in\mathbb{N}^+}$ (see (b1) above for the case $\ell=0$).
\par
For each $k>\ell$, we consider the finite horizon LQ  problem (\textbf{F-I-LQ})$_{\ell,k}$
 with $M=0$
and the same $(Q_j)_{j\in\mathbb{N}^+}$, $(R_j)_{j\in\mathbb{N}^+}$ as in the problem (\textbf{I-I-LQ})$_{\ell}$ with respect to the initial data $x_0\in H$ (for the definition of (\textbf{F-I-LQ})$_{\ell,k}$, one can see Section \ref{yu-sec-2.1} with $\widehat{k}=k$).  Then, it is obvious that
\begin{equation}\label{1016-2}
V(x_0;\ell,k)\leq J(u^{\ell};x_0,\ell)\leq V(x_0;\ell)\;\;\textrm{for each}\;k>\ell,
\end{equation}
  where $V(\cdot;\ell,k)$ is defined by (\ref{1014-9}) with $\widehat{k}=k$. Let $x^{\ell,k,x_0}(\cdot)$ and $u^{\ell,k,x_0}=(u^{\ell,k,x_0}_{j+\ell})_{j\in\mathbb{N}^+}\in l^2(\mathbb{N}^+;U)$ be the optimal state and the optimal control of the problem
(\textbf{F-I-LQ})$_{\ell,k}$ with respect to the initial state $x_0$. Then, by (\ref{1016-2}),
\begin{equation}\label{1016-3}
\sum_{j=\ell+1}^k(\langle Q_j x^{\ell,k,x_0}(t_j), x^{\ell,k,x_0}(t_j)\rangle_H
+\langle R_j u^{\ell,k,x_0}_j, u^{\ell,k,x_0}_j\rangle_U)\leq V(x_0;\ell) \;\;\mbox{for each}\;\;k\in\mathbb N^+.
\end{equation}
    Since the effective interval of (\textbf{F-I-LQ})$_{\ell,k}$ is $[t_{\ell},t_k]$, we permit that
\begin{equation}\label{yu-3-4-1}
    u^{\ell,k,x_0}_{j}=0
    \;\;\mbox{for each}\;\;j>k.
\end{equation}
    Define
\begin{equation}\label{yu-3-0-2}
    \widehat{x}^{\ell,k,x_0}_j
    =
\begin{cases}
     x^{\ell,k,x_0}(t_j),&\ell \leq j\leq k,\\
     0,&j>k,
\end{cases}
\end{equation}
    and $\widehat{x}^{\ell,k,x_0}:=(\widehat{x}^{\ell,k,x_0}_{j+\ell})_{j\in\mathbb{N}^+}$.
   Because  $(Q_j)_{j\in\mathbb{N}^+}\in \mathcal{M}_{\hbar,\gg}(H)$ and $(R_j)_{j\in\mathbb{N}^+}\in\mathcal{M}_{\hbar,\gg}(U)$, there exists a  constant $\delta>0$ so that
$Q_j-\delta I\in \mathcal {SL}_+(H)$ and $R_j-\delta I\in \mathcal {SL}_+(U)$ for each $j\in\mathbb{N}^+$.
By (\ref{yu-3-4-1}), (\ref{yu-3-0-2}) and \eqref{1016-3}, we have that
\begin{eqnarray}\label{1016-4}
V(x_0;\ell)&\geq&\sum_{j=\ell+1}^{+\infty}(\|Q_j^{\frac{1}{2}}\widehat{x}_j^{\ell,k,x_0}\|_H^2
+\|R_j^{\frac{1}{2}}u_j^{\ell,k,x_0}\|_U^2)\nonumber\\
&\geq&\delta\sum_{j=\ell+1}^{+\infty}\left(\|\widehat{x}^{\ell,k,x_0}_j\|_H^2+\|u_j^{\ell,k,x_0}\|_U^2
\right)\;\;\mbox{for each}\;\;k>\ell.
\end{eqnarray}
Thus,   there exist a subsequence of $(\widehat{x}^{\ell,k,x_0}, u^{\ell,k,x_0})_{k> \ell}$, still denoted by the same way, and  $(x^{\ell,x_0}:=(x_{j+\ell}^{\ell,x_0})_{j\in\mathbb{N}^+},
u^{\ell,x_0}:=(u_{j+\ell}^{\ell,x_0})_{j\in\mathbb{N}^+}) \in l^2(\mathbb{N}^+;H)
\times l^2(\mathbb{N}^+;U)$ so that
\begin{equation}\label{1016-5}
\begin{cases}
\widehat{x}^{\ell,k,x_0}\to x^{\ell,x_0} \textrm{ weakly in }  l^2(\mathbb{N}^+;H)\\
u^{\ell,k,x_0}\to u^{\ell,x_0} \textrm{ weakly in }  l^2(\mathbb{N}^+;U)
\end{cases}
    \textrm{ as }k\to\infty.
\end{equation}
    By (\ref{1016-5}) and (\ref{yu-3-0-2}), we have that, for each $j>\ell$,
\begin{eqnarray*}\label{yu-3-4-2}
    \langle x^{\ell,x_0}_j,\varphi\rangle_H&=&\lim_{k\to+\infty}\langle \widehat{x}^{\ell,k,x_0}_j,\varphi\rangle_H=\lim_{k\to+\infty}\langle x^{\ell,k,x_0}(t_j),\varphi\rangle_H\nonumber\\
    &=&\lim_{k\to+\infty}\left\langle
    \mathrm e^{A(t_j-t_\ell)}x_0+\sum_{i=\ell+1}^{j-1}\mathrm e^{A(t_j-t_i)}B_i u_{i}^{\ell,k,x_0}, \varphi\right\rangle_H\nonumber\\
    &=& \langle x(t_j;x_0,u^{\ell,x_0},\ell), \varphi\rangle_H\;\;\mbox{for any}\;\;\varphi\in H.
\end{eqnarray*}
    This implies that
\begin{equation}\label{yu-3-4-3}
    x^{\ell,x_0}=(x(t_{j+\ell};x_0,u^{\ell,x_0},\ell))_{j\in\mathbb{N}^+}.
\end{equation}
    It follows that
\begin{equation*}\label{yu-3-4-4}
    (x(t_{j+\ell};x_0,u^{\ell,x_0},\ell))_{j\in\mathbb{N}^+}\in l^2(\mathbb{N}^+;H)
\end{equation*}
    and
\begin{equation*}\label{yu-3-4-5}
    u^{\ell,x_0}\in\mathcal{U}_{ad}(x_0;\ell).
\end{equation*}
    Moreover, since $\mathcal{Q}\in \mathcal{M}_{\hbar,\gg}(H)$ and $\mathcal{R}\in\mathcal{M}_{\hbar,\gg}(U)$, by (\ref{1016-5}) and (\ref{yu-3-4-3}), we have
\begin{equation*}\label{1016-6}
\begin{cases}
    \|(Q_{j+\ell}^{\frac{1}{2}}x(t_{j+\ell};x_0,
    u^{\ell,x_0},\ell))_{j\in\mathbb{N}^+}\|_{l^2(\mathbb{N}^+;H)}\leq \liminf_{k\to\infty}\|(Q_{j+\ell}^{\frac{1}{2}}
    \widehat{x}^{\ell,k,x_0}_{j+\ell})_{j\in\mathbb{N}^+}
    \|_{l^2(\mathbb{N}^+;H)},\\
\|(R_{j+\ell}^{\frac{1}{2}}u_{j+\ell}^{\ell,x_0})_{j\in\mathbb{N}^+}\|_{l^2(\mathbb{N}^+;U)}\leq \liminf_{k\to\infty}\|(R_{j+\ell}^{\frac{1}{2}}u_{j+\ell}^{\ell,k,x_0})_{j\in\mathbb{N}^+}
\|_{l^2(\mathbb{N}^+;U)}.
\end{cases}
\end{equation*}
    This, together with the first inequality in (\ref{1016-4}), gives that
\begin{equation}\label{yu-3-4-6}
    J(u^{\ell,x_0};x_0,\ell)\leq \liminf_{k\to+\infty}V(x_0;\ell,k).
\end{equation}
    This, along with (\ref{1016-2}), yields that
\begin{equation*}\label{yu-3-4-7}
    J(u^{\ell,x_0};x_0,\ell)\leq V(x_0;\ell).
\end{equation*}
    By this, the arbitrariness of $x_0$, and the optimality of $V(x_0;\ell)$, we can conclude that $u^{\ell,x_0}$ is the optimal control to the problem (\textbf{I-I-LQ})$_{\ell}$ (i.e., $u^{\ell,x_0}=u^{\ell}$) and
\begin{equation*}\label{yu-3-4-8}
    J(u^{\ell,x_0};x_0,\ell)= V(x_0;\ell)\;\;\mbox{for any}\;\;x_0\in H.
\end{equation*}
    This, along with (\ref{yu-3-4-6}) and (\ref{1016-2}), gives that
\begin{equation}\label{yu-3-4-9}
    V(x_0;\ell)=\lim_{k\to+\infty} V(x_0;\ell,k).
\end{equation}
    For each $k>\ell$, let $\{P_j^k\}_{j=1}^k$ be the solution of the equation (\ref{1014-10}) with $\widehat k=k$. Thus, by  Proposition \ref{thm-filq} and Corollary \ref{yu-corollary-4-7-2}, we obtain
\begin{equation}\label{yu-3-4-10}
     V(x_0;\ell,k)=\langle P_\ell^kx_0,x_0\rangle_H=\|(P_{\ell}^k)^{\frac{1}{2}}x_0\|^2_H\;\;\mbox{for each}\;\;x_0\in H.
\end{equation}
    This, along with (\ref{1016-2}), implies that
\begin{equation*}\label{yu-3-4-11}
    \sup_{k>\ell} \|(P_{\ell}^k)^{\frac{1}{2}}x_0\|_H
    \leq \sqrt{V(x_0;\ell)}<+\infty\;\;\mbox{for each}\;\;x_0\in H.
\end{equation*}
    Thus, by  the uniform boundedness theorem, there is a constant $K>0$, so that
\begin{equation*}\label{yu-3-4-12}
    \|(P_\ell^k)^{\frac{1}{2}}\|_{\mathcal{L}(H)}\leq \sqrt{K}\;\;\mbox{for each}\;\;k>\ell.
\end{equation*}
    This, together with (\ref{yu-3-4-10}) again, implies that
\begin{equation}\label{yu-3-4-13-b}
    V(x_0;\ell,k)\leq K\|x_0\|^2_H\;\;\mbox{for each}\;\;x_0\in H.
\end{equation}
    Thus, by (\ref{yu-3-4-9}) and (\ref{yu-3-4-13-b}), we can conclude that (\ref{1016-1}) holds.

\vskip 5pt
\emph{Step 2.  We prove that $V(\cdot;\ell)$ satisfies the parallelogram law, i.e.,
\begin{equation}\label{1023-1}
V(x_0+y_0;\ell)+V(x_0-y_0;\ell)=2\left(V(x_0;\ell)+V(y_0;\ell)\right),\,\forall\,
x_0,y_0\in H.
\end{equation}}
\par
In fact, by Proposition \ref{thm-filq}, for each $k>\ell$, it is obvious that
\begin{equation*}\label{yu-3-4-13}
    V(x_0+y_0;k,\ell)+V(x_0-y_0;k,\ell)=2(V(x_0;k,\ell)+V(y_0;k,\ell))\;\;\textrm{ for any }x_0, y_0\in H.
\end{equation*}
   This, together with (\ref{yu-3-4-9}), yields that (\ref{1023-1}) holds.
\vskip 5pt
    \emph{Step 3. We complete the proof of $(i)$.}
\par
     Indeed, by (\ref{1016-1}), we can conclude that $V(\cdot;\ell)$ is continuous in $0$. Thus, by (\ref{1023-1}) and \cite[Theorem 2]{Kurepa}, $V(\cdot;\ell)$ is continuous at any point. Thus, by (\ref{1023-1}) again and  \cite[Theorem 3]{Kurepa}, there is a unique symmetric
    and bounded linear operator $P_\ell$ so that
\begin{equation*}\label{yu-3-4-14}
    V(x_0;\ell)=\langle P_\ell x_0,x_0\rangle_H\;\;\mbox{for each}\;\;x_0\in H.
\end{equation*}
    By the fact that $V(x_0;\ell)\geq 0$ for each $x_0\in H$, it is clear that $P_\ell\in\mathcal{SL}_{+}(H)$.
Hence,  we complete the proof of $(i)$.
\par
   Next, we give the proof of $(ii)$.
    To this end, by (\ref{1015-2}), it suffices to show that
\begin{equation}\label{1024-1}
V(x_0;\ell)=V(x_0;\ell+\hbar)\,\textrm{ for each } x_0\in H.
\end{equation}
In fact, let $x_0\in H$ be arbitrarily  fixed. Let  $u^\ell=(u^\ell_{j+\ell})_{j\in\mathbb{N}^+}\in l^2(\mathbb{N}^+;U)$ be the optimal control of (\textbf{I-I-LQ})$_\ell$ with respect to
the initial state $x_0$, i.e.,
\begin{equation}\label{1024-2}
V(x_0;\ell)=J(u^\ell;x_0,\ell).
\end{equation}
Let $u^{\ell+\hbar}=(u^{\ell+\hbar}_{j+\ell+\hbar})_{j\in\mathbb{N}^+}\in l^2(\mathbb{N}^+;U)$ be defined as
\begin{equation}\label{1024-3}
u^{\ell+\hbar}_j:=u^\ell_{j-\hbar}, \,\forall\,j>\ell+\hbar.
\end{equation}
   By this and (\ref{Intro-2}), it is easy to check that
\begin{equation}\label{1024-4}
x(t_j;x_0,u^{\ell+\hbar},\ell+\hbar)=x(t_{j-\hbar};x_0,u^\ell,\ell)\;\; \mbox{for each}\;\;j\geq \ell+ \hbar.
\end{equation}
Thus, by (\ref{yu-b-3-4-1}), (\ref{1024-3}) and (\ref{1024-4}), we have
\begin{equation*}\label{1024-5}
V(x_0;\ell+\hbar)\leq J(u^{\ell+\hbar};x_0,\ell+\hbar)=J(u^\ell;x_0,\ell).
\end{equation*}
    This, along with (\ref{1024-2}), implies that
    $V(x_0;\ell+\hbar)\leq V(x_0;\ell)$. Similarly, it holds that $V(x_0;\ell)\leq V(x_0;\ell+\hbar)$. Therefore,  \eqref{1024-1} holds. Thus, by the arbitrariness of $x_0$ and \eqref{1024-1}, the claim $(ii)$ is true.
\par
This completes the proof.
\end{proof}
The following result is the dynamic programming  principle to  the problem \textbf{(I-I-LQ)}$_{\ell}$, which has been proved in \cite{Q-W-Y} for finite dimensional case. One can easily check that, by the same arguments used in \cite{Q-W-Y}, the dynamic programming  principle also holds for infinite dimensional setting. Thus, we omit its proof.

\begin{lemma}\label{lem-3-3}
For each $k>\ell$ ($\ell, k\in\mathbb N$), it holds that
%\begin{equation}\label{1025-1}
%\begin{split}
% V(x_0;\ell)=\inf_{\substack{u=(u_j)_{j>\ell}\\u\in l^2(\mathbb N^+;(L^2(\Omega))^m) }}
%\Big\{&\sum_{j=\ell+1}^k \left(\langle Q_j x(t_j;x_0,u,\ell), x(t_j;x_0,u,\ell)  \rangle  +\langle R_j u_j,u_j\rangle\right)\\
%& +V(x(t_k^+;x_0,u,\ell);k)\Big\},\,\forall\,x_0\in(L^2(\Omega))^n.
%\end{split}
%\end{equation}
\begin{eqnarray*}\label{1025-1}
 V(x_0;\ell)=\inf_{u=(u_{j+
 \ell})_{j\in\mathbb{N}^+}\in l^2(\mathbb{N}^+;U) }
\Big\{\sum_{j=\ell+1}^k \left(\langle Q_j x(t_j;x_0,u,\ell), x(t_j;x_0,u,\ell)  \rangle_H  +\langle R_j u_j,u_j\rangle_U\right)\nonumber\\
  +V(x(t_k^+;x_0,u,\ell);k)\Big\}\;\;\mbox{for any}\;\;x_0\in H.
\end{eqnarray*}

\end{lemma}

%\begin{proof}
%It is easy.
%\end{proof}

The results in  Proposition
 \ref{thm-filq}, Lemma \ref{lem-3-2} and Lemma \ref{lem-3-3} lead to the following proposition.

\begin{proposition}\label{thm-3}
  Let the operators $ P_\ell\in \mathcal{SL}_{+}(H)$ be so that (\ref{1015-2}) in Lemma \ref{lem-3-2} holds for each $\ell=0,1,\ldots,\hbar$. Then $\{P_\ell\}_{\ell=0}^\hbar$ satisfy the Riccati-type equation \eqref{1014-3}.
\end{proposition}

\begin{proof}
For each $0\leq \ell<\hbar$, by Lemma \ref{lem-3-3} and the claim $(i)$ in Lemma \ref{lem-3-2}, we have that
\begin{eqnarray}\label{0204-4}
 \langle P_\ell x_0, x_0 \rangle_H=\inf_{u=(u_{j+\ell})_{j\in\mathbb{N}^+}\in l^2(\mathbb{N}^+;U)}
\Big\{\sum_{j=\ell+1}^\hbar \left(\langle Q_j x(t_j;x_0,u,\ell), x(t_j;x_0,u,\ell)  \rangle_H  +\langle R_j u_j,u_j\rangle_U\right)\nonumber\\
 +\langle P_\hbar x_0, x_0 \rangle_H\Big\}\;\;\mbox{for any}\;\;x_0\in H.
\end{eqnarray}
 It is obvious that  the right side of (\ref{0204-4}) is a finite horizon LQ problem {\textbf{(F-I-LQ)}}$_{\ell,\hbar}$ with $M=P_\hbar$ (for the definition of (\textbf{F-I-LQ})$_{\ell,\hbar}$, one can see Section \ref{yu-sec-2.1} with $\widehat{k}=\hbar$). Then, from Proposition
 \ref{thm-filq}, we can claim that
   $\{P_\ell\}_{\ell=1}^\hbar$ verifies the following Riccati-type equation:
\begin{eqnarray*}\label{yu-3-11-1}
&\;&P_{\ell-1}-\mathrm e^{A^*(t_{\ell}-t_{\ell-1})}P_{\ell}\mathrm e^{A(t_{\ell}-t_{\ell-1})}\nonumber\\
&=&\mathrm e^{A^*(t_{\ell}-t_{\ell-1})}Q_{\ell}\mathrm e^{A(t_{\ell}-t_{\ell-1})}
-\mathrm e^{A^*(t_{\ell}-t_{\ell-1})}P_{\ell}B_{\ell}(R_{\ell}+B_{\ell}^* P_{\ell}B_{\ell})^{-1} B_{\ell}^*P_{\ell}\mathrm e^{A(t_{\ell}-t_{\ell-1})}
\end{eqnarray*}
   for each $\ell\in\{1,2,\ldots,\hbar\}$. This, together with the claim $(ii)$ in Lemma \ref{lem-3-2}, yields that $\{P_\ell\}_{\ell=0}^k$ satisfies \eqref{1014-3}.
\end{proof}
\section{Proof of the main theorems}
    In this section, we shall prove Theorem \ref{yu-main-theorem} and Theorem
    \ref{thm-1}. As stated in Section \ref{intro}, the proof of Theorem \ref{yu-main-theorem} is based on Theorem \ref{thm-1}. Thus, we first prove Theorem \ref{thm-1} and then continue to prove Theorem \ref{yu-main-theorem}.
\subsection{The proof of Theorem \ref{thm-1}}

    We divide our proof into 4 steps.
\vskip 5pt
\emph{Step 1. The proof of $(i)\Rightarrow (ii)$.}
\par
Suppose the system $[A, \{B_k\}_{k=1}^\hbar,\Lambda_\hbar]$ is exponentially $\hbar$-stabilizable, i.e.,  there
exists a feedback law
 $$\mathcal F=\{F_k\}_{k=1}^\hbar\subset \mathcal L(H;U),$$
 so that  the solution $x_\mathcal F (t)$ of the closed-loop system (\ref{Intro-4(2)})
satisfies that
\begin{equation}\label{1103-2}
\|x_\mathcal F(t)\|_H \leq C \mathrm e^{-\mu t}\|x_\mathcal F(0)\|_H, \;\;\forall\,t>0
\end{equation}
for some $C>0$ and $\mu>0$. We denote the solution of (\ref{Intro-4(2)}) with initial data $x_0$ by $x_{\mathcal{F}}(\cdot;x_0)$.  For any $x_0\in H$, we take $u=(u_j)_{j\in\mathbb N^+}$ in (\ref{Intro-3}) (with $x(0)=x_0$) with
\begin{equation}\label{yu-3-5-3}
 u_j:=F_{\nu(j)} x(t_j;x_0,u),\;\;\forall\,j\in\mathbb{N}^+.
\end{equation}
    One can easily check that  $x_{\mathcal{F}}(\cdot;x_0)=x(\cdot;x_0,u)$.
Then, by (\ref{1103-2}), we have
\begin{eqnarray}\label{1103-3}
\sum_{j=1}^{+\infty} \|x(t_j;x_0,u)\|^2_H &\leq& C^2 \|x_0\|^2_H \left(\sum_{k=0}^{+\infty} \mathrm e^{-2k\mu t_\hbar}\right)\left(\sum_{j=1}^\hbar \mathrm e^{-2\mu t_j}\right)\nonumber\\
&\leq& C^2 \|x_0\|^2_H\left(\sum_{j=1}^\hbar \mathrm e^{-2\mu t_j}\right)\frac{1}{1- \mathrm e^{-2\mu t_\hbar}}<+\infty.
\end{eqnarray}
    This, together with (\ref{yu-3-5-3}), gives that
\begin{equation}\label{1103-4}
\sum_{j=1}^{+\infty }\|u_j\|^2_U\leq \max_{1\leq j\leq \hbar} \|F_j\|^2_{\mathcal L(U;H)}
\sum_{j=1}^{+\infty}\|x(t_j;x_0,u)\|^2_H<+\infty.
\end{equation}
  Thus, by (\ref{1103-3}) and (\ref{1103-4}), we have $u\in \mathcal U_{ad}(x_0)$, which indicates that $\mathcal U_{ad}(x_0)\neq \emptyset$. Hence, the claim $(ii)$ is true.

\vskip 5pt
    \emph{Step 2. The proof of $(ii)\Rightarrow (iii)$.}
\par
   Suppose the claim $(ii)$ is true. Let $(Q_j)_{j\in\mathbb N^+}\in\mathcal M_{\hbar,\gg}(H)$ and $(R_j)_{j\in\mathbb N^+}\in\mathcal M_{\hbar,\gg}(U)$ be arbitrarily fixed. From Proposition \ref{thm-3}, we know that the claim $(ii)$ implies that the Riccati-type equation \eqref{1014-3} has a solution $\{P_\ell\}_{\ell=0}^\hbar$.
We only need to prove this solution  is unique. Suppose $\{\widetilde P_\ell\}_{\ell=0}^\hbar$ is  a solution
of the Riccati-type equation \eqref{1014-3}. By Lemma \ref{lem-3-2}, it suffices to prove that
\begin{equation}\label{yu-3-5-5-1}
V(x_0;\ell)=\langle \widetilde P_\ell x_0, x_0\rangle_H, \,\forall\,x_0\in H,\,\forall\, 0\leq \ell\leq \hbar,
\end{equation}
    where $V(\cdot;\ell)$ is defined by (\ref{1015-9}).
\par
Let $x_0\in H$ and $0\leq \ell\leq \hbar$ be arbitrarily fixed.
Then, for any $u=(u_{j+\ell})_{j\in\mathbb N^+}\in \mathcal U_{ad}(x_0;\ell)$ and $k> \ell$,
it stands that
\begin{equation}\label{1105-1}
x(t_{k};x_0,u,\ell)=\mathrm e^{A(t_{k}-t_{k-1})}x(t_{k-1}^+;x_0,u,\ell)
\end{equation}
and
\begin{equation}\label{1105-1-1}
x(t_{k}^+;x_0,u,\ell)=\mathrm e^{A(t_{k}-t_{k-1})}x(t_{k-1}^+;x_0,u,\ell)+B_{k} u_{k},
\end{equation}
    where $x(\cdot;x_0,u,\ell)$ is the solution of the equation (\ref{1014-8}) with $v=u$.
Denote $x(\cdot):=x(\cdot;x_0,u,\ell)$ for simplicity.
By \eqref{1105-1}, (\ref{Intro-2}) and  the fact that $\{\widetilde P_\ell\}_{\ell=0}^\hbar$ satisfies  the equation \eqref{1014-3}, we obtain that, for each $k>\ell$,
\begin{eqnarray*}\label{1105-2-b}
&&\langle \widetilde P_{\nu(k-1)} x(t_{k-1}^+), x(t_{k-1}^+)\rangle_H -\langle \widetilde P_{\nu(k)} \mathrm e^{A(t_{k}-t_{k-1})}x(t_{k-1}^+), \mathrm e^{A(t_{k}-t_{k-1})}x(t_{k-1}^+)\rangle_H\nonumber\\
&=& \langle Q_{k}\mathrm e^{A(t_{k}-t_{k-1})}x(t_{k-1}^+),  \mathrm e^{A(t_{k}-t_{k-1})}x(t_{k-1}^+)\rangle_H\nonumber\\
&&-\langle (R_{k}+B_{\nu(k)}^*\widetilde P_{\nu(k)}B_{\nu(k)})^{-1}
                      B_{\nu(k)}^*\widetilde P_{\nu(k)}\mathrm e^{A(t_{k}-t_{k-1})}x(t_{k-1}^+),B_{\nu(k)}^*\widetilde P_{\nu(k)}\mathrm e^{A(t_{k}-t_{k-1})}x(t_{k-1}^+)\rangle_U\nonumber\\
&=&\langle Q_{k}x(t_{k}),  x(t_{k})\rangle_H
      -\langle (R_{k}+B_{\nu(k)}^*\widetilde P_{\nu(k)}B_{\nu(k)})^{-1} B_{\nu(k)}^*\widetilde P_{\nu(k)}x(t_{k}),B_{\nu(k)}^*\widetilde P_{\nu(k)}x(t_{k})\rangle_U.
\end{eqnarray*}
This, along with \eqref{1105-1-1}, gives that
\begin{eqnarray}\label{1105-2}
&&\langle Q_{k}x(t_{k}),  x(t_{k})\rangle_H+\langle R_{k} u_{k}, u_{k}\rangle_U\nonumber\\
&=&\langle \widetilde P_{\nu(k-1)} x(t_{k-1}^+), x(t_{k-1}^+)\rangle_H-\langle\widetilde  P_{\nu(k)} x(t_{k}^+), x(t_{k}^+)\rangle_H
+\langle (R_{k}+B_{\nu(k)}^*\widetilde P_{\nu(k)} B_{\nu(k)})u_{k}, u_{k}\rangle_U \nonumber\\
&&+\langle (R_{k}+B_{\nu(k)}^*\widetilde P_{\nu(k)}B_{\nu(k)})^{-1} B_{\nu(k)}^* \widetilde P_{\nu(k)}x(t_{k}),B_{\nu(k)}^*\widetilde P_{\nu(k)}x(t_{k}) \rangle_U
    + 2\langle \widetilde P_{\nu(k)} B_{\nu(k)} u_{k}, x(t_{k}) \rangle_U\nonumber\\
&=&\langle \widetilde P_{\nu(k-1)} x(t_{k-1}^+), x(t_{k-1}^+)\rangle_H -\langle \widetilde P_{\nu(k)} x(t_{k}^+), x(t_{k}^+) \rangle_H \nonumber\\
&&+\left\|(R_{k}+B_{\nu(k)}^*\widetilde P_{\nu(k)} B_{\nu(k)})^{\frac{1}{2}}
     \left(u_{k}+(R_{k}+B_{\nu(k)}^*\widetilde P_{\nu(k)} B_{\nu(k)})^{-1} B_{\nu(k)}^* \widetilde P_{\nu(k)} x(t_{k}) \right) \right\|_U^2.
\end{eqnarray}
Since $u\in \mathcal U_{ad}(x_0)$, one can easily check that  $\lim_{j\to+\infty} \langle \widetilde{P}_{\nu(j)}x(t^+_j), x(t_j^+)\rangle_H=0$. Therefore, by (\ref{yu-b-3-4-1}) and \eqref{1105-2}, we can conclude that, for any $u=(u_{j+\ell})_{j\in\mathbb N^+}\in \mathcal U_{ad}(x_0;\ell)$,
\begin{eqnarray}\label{1105-3}
&\;&J(u;x_0,\ell)\nonumber\\
&=&\sum_{k=\ell+1}^{+\infty} \left( \langle Q_{k}x(t_{k}),  x(t_{k})\rangle_H+\langle R_{k} u_{k}, u_{k}\rangle_U \right)\nonumber\\
&=&\langle \widetilde P_\ell x_0, x_0\rangle_H \nonumber\\
&&+\sum_{k=\ell+1}^{+\infty}
 \left\|(R_k+B_{\nu(k)}^*\widetilde P_{\nu(k)} B_{\nu(k)})^{\frac{1}{2}}\left(u_{k}+(R_k+B_{\nu(k)}^* \widetilde P_{\nu(k)} B_{\nu(k)})^{-1} B_{\nu(k)}^*\widetilde  P_{\nu(k)} x(t_{k}) \right)\right \|^2_U.\quad
 \end{eqnarray}
 This, together with the definition of $V(\cdot;\ell)$, implies that
\begin{equation}\label{yu-3-5-7}
     V(x_0;\ell)\geq \langle \widetilde P_\ell x_0,x_0\rangle_H.
\end{equation}
    Further, we define $u^*=(u^*_{j+\ell})_{j\in\mathbb{N}^+}$ as
$$
    u^*_k:=-(R_k+B_{\nu(k)}^* \widetilde P_{\nu(k)} B_{\nu(k)})^{-1} B_{\nu(k)}^*\widetilde  P_{\nu(k)} x(t_{k};x_0,u^*,\ell)\;\;\mbox{for any}\;\;k\in\mathbb{N}^+.
$$
     By similar  proofs of (\ref{1105-3}), one can show that
\begin{equation}\label{yu-3-5-8}
    J(u^*;x_0,\ell)=\langle \widetilde{P}_\ell x_0,x_0\rangle_H.
\end{equation}
    This, along with the facts $(Q_j)_{j\in\mathbb N^+}\in\mathcal M_{\hbar,\gg}(H)$ and $(R_j)_{j\in\mathbb N^+}\in\mathcal M_{\hbar,\gg}(U)$, yields that
\begin{equation*}\label{yu-3-5-9}
    (x(t_{j+\ell};x_0,u^*,\ell))_{j\in\mathbb{N}^+} \in l^2(\mathbb{N}^+;H)
    \;\;\mbox{and}\;\;u^*\in l^2(\mathbb{N}^+;U).
\end{equation*}
    These imply that $u^*\in \mathcal{U}_{ad}(x_0;\ell)$ and by (\ref{yu-3-5-8}),
\begin{equation}\label{yu-3-5-10}
    V(x_0;\ell)\leq \langle \widetilde{P}_\ell x_0,x_0\rangle_H.
\end{equation}
    Therefore, by (\ref{yu-3-5-7}), (\ref{yu-3-5-10}) and the arbitrariness of $x_0$ and $\ell$, we have (\ref{yu-3-5-5-1}), i.e., the claim $(iii)$ holds.
\vskip 5pt
    \emph{Step 3. The proof of $(iii)\Rightarrow (iv)$.} It is trivial.
\vskip 5pt
    \emph{Step 4. The proof of $(iv)\Rightarrow(i)$.}
\par
    Given $(Q_j)_{j\in\mathbb N^+}\in \mathcal{M}_{\hbar,\gg}(H)$ and $(R_j)_{j\in\mathbb N^+}\in\mathcal{M}_{\hbar,\gg}(U)$ so that the claim $(iv)$ is true.
Then, by the definitions of $\mathcal{M}_{\hbar,\gg}(H)$ and $\mathcal{M}_{\hbar,\gg}(U)$, there exist constants $q>0$ and $\delta>0$ so that
\begin{equation}\label{yu-3-5-bb-1}
    Q_j-q I>0\;\;\mbox{and}\;\; R_j-\delta I>0\;\;\mbox{for each}\;\;j\in\mathbb N^+.
\end{equation}
Let  $\{P_\ell\}_{\ell=1}^\hbar$ be  the unique solution of the equation \eqref{1014-3}.
Consider the following equation
\begin{equation}\label{1109-1}
\begin{cases}
  \displaystyle x'(t)=Ax(t),   & t\in\mathbb R^+\backslash \Lambda_\hbar,   \\
  x(t_j^+)=x(t_j)+B_{\nu(j)} F_{\nu(j)} x(t_j),   &j\in\mathbb{N}^+,   \\
\end{cases}
\end{equation}
where
\begin{equation}\label{yu-bb-3-7-h}
F_{j}:=-(R_j+B_{j}^*P_jB_{j})^{-1}B_{j}^*P_j,\;\;\forall\,1\leq j \leq \hbar.
\end{equation}
    Denote the solution of \eqref{1109-1} by $x_\mathcal F(\cdot;x_0)$
with initial data $x_0\in H$. Let $\Phi_{\hbar}:\,H\to H$ be defined as
\begin{equation}\label{1109-2}
\Phi_{\hbar}:= (I+B_\hbar F_{\hbar})\mathrm e^{A(t_{\hbar}-t_{\hbar-1})}\cdots (I+B_{2}F_{2})\mathrm e^{A(t_2-t_1)}(I+B_1 F_1)\mathrm e^{At_1}.
\end{equation}
   It is clear that  $\Phi_{\hbar}\in \mathcal{L}(H)$.     By (\ref{1109-2}) and (\ref{Intro-2}), one can easily check that, for any $x_0\in H$,
\begin{equation}\label{yu-bb-3-7-1}
x_\mathcal F(t_{k\hbar}^+;x_0)=(\Phi_{\hbar})^k x_0 \textrm{ for each }k\in\mathbb N^+.
\end{equation}
By (\ref{yu-3-5-bb-1}) and the same arguments as those lead to \eqref{yu-3-5-10}, we know that
\begin{eqnarray*}
\sum_{k=1}^{+\infty} \|x_\mathcal F(t_{k\hbar}^+;x_0) \|^2_H
&\leq&
\|I+B_\hbar F_\hbar\|_{\mathcal L(H) }^2\sum_{k=1}^{+\infty}\|x_\mathcal F(t_{k\hbar};x_0) \|_H^2 \\
&\leq& \frac{1}{q}\|(I+B_\hbar F_\hbar)\|_{\mathcal L(H)}^2\langle P_0 x_0,x_0\rangle_H
    \;\;\mbox{for any}\;\;x_0\in H.
\end{eqnarray*}
    This, together with (\ref{yu-bb-3-7-1}), yields that
\begin{equation}\label{yu-5-6-2}
\begin{split}
\sum_{k=1}^{+\infty} \|(\Phi_{\hbar})^k x_0\|^2_H
\leq C_{\hbar,q} \|x_0\|^2_H
\;\;\mbox{for any}\;\;x_0\in H,
\end{split}
\end{equation}
    where $C_{\hbar,q}:=\frac{1}{q}  \|I+B_\hbar F_\hbar\|_{\mathcal L(H)}^2 \|P_0\|_{\mathcal L(H)}$. It follows that
\begin{equation}\label{0205-4}
    \sup_{k\in\mathbb{N}^+}\|(\Phi_{\hbar})^k\|_{\mathcal{L}(H)}\leq \sup_{\|x_0\|_H=1}
    \left(\sum_{k=1}^{+\infty}\|(\Phi_{\hbar})^kx_0\|^2_H\right)^{\frac{1}{2}}
    \leq\sqrt{C_{\hbar,q}}.
\end{equation}
    Thus, by (\ref{yu-5-6-2}) and (\ref{0205-4}), for each $k\in\mathbb{N}^+$ and any $x_0\in H$,
\begin{eqnarray}\label{yu-5-6-1}
    k\|(\Phi_{\hbar})^kx_0\|^2_H&=&\sum_{j=1}^k\|(\Phi_{\hbar})^kx_0\|_H^2
    \leq \sum_{j=1}^k\|(\Phi_{\hbar})^{k-j}\|_{\mathcal{L}(H)}^2\|(\Phi_{\hbar})^jx_0\|_H^2\nonumber\\
    &\leq&C_{\hbar,q}\sum_{j=1}^k\|(\Phi_{\hbar})^jx_0\|_H^2\leq C_{\hbar,q}^2\|x_0\|^2.
\end{eqnarray}
    Let $k_{\hbar,q}:=[4C_{\hbar,q}^2]+1$. By (\ref{yu-5-6-1}), we have
\begin{equation}\label{yu-5-6-3}
    \|(\Phi_{\hbar})^{k}\|_{\mathcal{L}(H)}\leq \frac{1}{2}\;\;\mbox{for each}\;\;k\geq k_{\hbar,q}.
\end{equation}
    Fix $\widehat{k}\geq k_{\hbar,q}$. Then, by (\ref{yu-5-6-3}), we have $\|(\Phi_{\hbar})^{\widehat{k}}\|_{\mathcal{L}(H)}\leq \frac{1}{2}$. It follows that
\begin{equation*}\label{yu-3-7-8}
     \|(\Phi_{\hbar})^{n\widehat{k}}\|_{\mathcal{L}(H)}\leq \left(\frac{1}{2}\right)^n
     \;\;\mbox{for each}\;\;n\in\mathbb{N}^+.
\end{equation*}
    This gives that
\begin{equation*}\label{yu-3-7-9}
    \limsup_{n\to+\infty}\|(\Phi_{\hbar})^{n\widehat{k}}
    \|_{\mathcal{L}(H)}^{\frac{1}{n\widehat{k}}}\leq \left(\frac{1}{2}\right)^{\frac{1}{\widehat{k}}}<1.
\end{equation*}
    This implies that, there exists a $n^*\in \mathbb{N}^+$ so that
\begin{equation*}\label{yu-3-7-10}
    \|(\Phi_{\hbar})^{n^*\hat{k}}\|_{\mathcal{L}(H)}\leq e^{-\alpha t_{n^*\widehat{k}\hbar}},
\end{equation*}
    where $\alpha=\alpha(n^*,\widehat{k}):=t^{-1}_{n^*\widehat{k}\hbar}\ln\left(2\left(1+\left(\frac{1}{2}
    \right)^{\widehat{k}}\right)\right)(>0)$.
    This, together with (\ref{yu-bb-3-7-1}) and
    (\ref{Intro-2}), yields that, for any $x_0\in H$,
\begin{equation}\label{yu-3-7-11}
    \|x_{\mathcal{F}}(t^+_{jn^*\widehat{k}\hbar};x_0)\|_H
    \leq e^{-j\alpha t_{n^*\widehat{k}\hbar}}\|x_0\|_H
    =e^{-\alpha t_{jn^*\widehat{k}\hbar}}\|x_0\|_H
    \;\;\mbox{for each}\;\;j\in\mathbb{N}^+.
\end{equation}
    One can easily check that, for any $x_0\in H$,
\begin{equation}\label{yu-3-7-12}
    \|x_{\mathcal F}(t;x_0)\|_H\leq C_{n^*,\widehat{k}}\|x_0\|_H\;\;\mbox{for all}\;\;t\in[0,t_{n^*\widehat{k}\hbar}],
\end{equation}
    where $C_{n^*,\widehat{k}}>0$ is a constant.  Let $t\in\mathbb{R}^+$ be arbitrarily fixed. It is clear that  there is $j^*\in \mathbb{N}$ so that $t_{j^*n^*\widehat{k}\hbar}<t\leq
    t_{(j^*+1)n^*\widehat{k}\hbar}$. Thus, by (\ref{yu-3-7-11}) and (\ref{yu-3-7-12}), we have that, for any $x_0\in H$,
\begin{eqnarray*}\label{yu-3-7-13}
    \|x_{\mathcal F}(t; x_0)\|_H
&=&\|x_{\mathcal F}(t-t_{j^*n^*\widehat{k}\hbar}; x_{\mathcal F}(t_{j^*n^*\widehat{k}\hbar}^+;x_0))\|_H
\leq C_{n^*,\widehat{k}}\|x_{\mathcal F}(t_{j^*n^*\widehat{k}\hbar}^+;x_0)\|_H\nonumber\\
&\leq&C_{n^*,\widehat{k}}\mathrm e^{\alpha (t-t_{j^*n^*\widehat{k}\hbar})}
\mathrm e^{-\alpha t}\|x_0\|_H
\leq C_{n^*,\widehat{k}}\mathrm e^{\alpha t_{n^*\widehat{k}\hbar}}
\mathrm e^{-\alpha t}\|x_0\|_H.
\end{eqnarray*}
    This, together with the arbitrariness of $t$, implies the system (\ref{Intro-3}) is exponentially stable with the feedback law
    $\mathcal{F}$ defined by (\ref{yu-bb-3-7-h}). Thus, the claim $(i)$ holds and the corresponding feedback law can be given as (\ref{1014-6}) (see (\ref{yu-bb-3-7-1})).
\par
  In summary, we complete the proof.

\subsection{The proof of Theorem \ref{yu-main-theorem}}

    The proof is divided into 4 steps.
\vskip 5pt
    \emph{Step 1. The proof of $(i)\Rightarrow (ii)$.}
\par
    Suppose that the system $[A,\{B_k\}_{k=1}^\hbar, \Lambda_\hbar]$ is exponentially $\hbar$-stabilizable.
    Then, there is a feedback law $\mathcal{F}=\{F_j\}_{j=1}^{\hbar}\subset \mathcal{L}(H;U)$
    so that the solution $x_{\mathcal{F}}(\cdot)$ of the closed-loop system (\ref{Intro-4(2)}) verifies
\begin{equation}\label{yu-3-8-2}
    \|x_{\mathcal{F}}(t)\|_H\leq C_1\mathrm e^{-\mu t}\|x_{\mathcal{F}}(0)\|_H
    \;\;\mbox{for any}\;\;t\in\mathbb{R}^+,
\end{equation}
    where $C_1$ and $\mu$ are two positive constants.
    We fix the feedback law $\mathcal{F}$. Since $x_{\mathcal{F}}(t_j^+)=(I+B_{\nu(j)}F_{\nu(j)})x_{\mathcal{F}}(t_j)$
     for each $j\in\mathbb{N}^+$, by (\ref{yu-3-8-2}), we have
\begin{equation}\label{yu-3-8-3}
    \|x_{\mathcal{F}}(t_j^+)\|_H\leq C_2 \mathrm e^{-\mu t_j}\|x_{\mathcal{F}}(0)\|_H
    \;\;\mbox{for any}\;\;j\in\mathbb{N}^+,
\end{equation}
    where $C_2:=C_1\sup_{1\leq k\leq \hbar}\|I+B_kF_k\|_{\mathcal{L}(H)}$.
\par
    Let $\sigma\in(0,1)$ and $x_0\in H$ be arbitrarily fixed. It is obvious that there is a
    $\widehat{k}=\widehat{k}(\sigma)\in\mathbb{N}^+$ so that
\begin{equation}\label{yu-3-8-4}
    C_2\mathrm e^{-\mu t_{\widehat{k}}}\leq \sigma.
\end{equation}
     We denote the closed-loop system (\ref{Intro-4(2)}) with the initial data $x_0$ by $x_{\mathcal{F}}(\cdot;x_0)$.
    Let
    $u_j(x_0):=F_{\nu(j)}x_{\mathcal{F}}(t_j;x_0)$ for each $j\in\mathbb{N}$ and $u(x_0):=(u_j(x_0))_{j\in\mathbb{N}^+}$ in the control system (\ref{Intro-3}). By (\ref{yu-3-8-2}), one can easily check that
\begin{equation}\label{yu-3-8-10}
    \|u_j(x_0)\|_U\leq C_3 \mathrm e^{-\mu t_j}\|x_0\|_H \;\;\mbox{for each}\;\;j\in\mathbb{N}^+,
\end{equation}
    where $C_3:=C_1\max_{1\leq k\leq \hbar}\|F_k\|_{\mathcal{L}(H;U)}$. It should be noted that, by (\ref{Intro-2}),
\begin{equation*}\label{yu-3-8-11}
    \sum_{j=1}^{+\infty}\mathrm e^{-2\mu t_j}\leq \max_{1\leq k\leq \hbar}\frac{1}{|t_k-t_{k-1}|}
    \int_0^{+\infty}\mathrm e^{-2\mu s}ds =\frac{1}{2\mu}\max_{1\leq k\leq \hbar} \frac{1}{|t_k-t_{k-1}|}.
\end{equation*}
    This, along with (\ref{yu-3-8-10}), yields that
\begin{equation}\label{yu-3-8-12}
    \|u(x_0)\|_{l^2(\mathbb{N}^+;U)}=\left(\sum_{j=1}^{+\infty}\|u_j(x_0)\|^2_U\right)^{\frac{1}{2}}
    \leq C_4\|x_0\|_H\;\;\mbox{for any}\;\;x_0\in H,
\end{equation}
    where $C_4:=C_3\left(\frac{1}{2\mu}\max_{1\leq k\leq \hbar}\frac{1}{|t_k-t_{k-1}|}\right)^{\frac{1}{2}}$.  Moreover, one can easily check that
\begin{equation*}\label{yu-3-8-5}
    x(t_k^+;x_0,u(x_0))=x_{\mathcal{F}}(t_k^+;x_0)\;\;\mbox{for each}\;\;k\in\mathbb{N}^+.
\end{equation*}
    This, together with (\ref{yu-3-8-3}) and (\ref{yu-3-8-4}), implies that
\begin{equation}\label{yu-3-8-6}
    \|x(t_{\widehat{k}}^+;x_0,u(x_0))\|_H\leq \sigma \|x_0\|_H.
\end{equation}
    Furthermore, it is obvious that
\begin{equation*}\label{yu-3-8-7}
    x(t_{\widehat{k}}^+;x_0,u(x_0))=\mathrm e^{At_{\widehat{k}}}x_0+\sum_{j=1}^{\widehat{k}}
    \mathrm e^{A(t_{\widehat{k}}-t_j)} B_{\nu(j)}u_j(x_0).
\end{equation*}
    Hence,
\begin{equation*}\label{yu-3-8-13}
    \langle x(t^+_{\widehat{k}};x_0,u(x_0)),\varphi\rangle_H
    =\langle x_0, \mathrm{e}^{A^*t_{\widehat{k}}}\varphi\rangle_H+\sum_{j=1}^{\widehat{k}}\langle u_j(x_0),B^*_{\nu(j)}\mathrm{e}^{A^*(t_{\widehat{k}}-t_j)}\varphi\rangle_U\;\;\mbox{for any}
    \;\;\varphi\in H.
\end{equation*}
    This, along with (\ref{yu-3-8-12}) and (\ref{yu-3-8-6}), yields that
\begin{equation*}\label{yu-3-8-14}
    \langle x_0, \mathrm e^{A^*t_{\widehat{k}}}\varphi\rangle_H
    \leq \sigma\|x_0\|_H\|\varphi\|_H+C_4\|x_0\|_H
    \left(\sum_{j=1}^{\widehat{k}}\|B^*_{\nu(j)}
    \mathrm{e}^{A^*(t_{\widehat{k}}-t_j)}\varphi\|_H^2\right)^{\frac{1}{2}}
    \;\;\mbox{for any}\;\;\varphi\in H.
\end{equation*}
    This, together with the arbitrariness of $x_0\in H$, implies (\ref{yu-3-8-1}) with $C(\sigma):=C_4$. Therefore, the claim $(ii)$ is true.
\vskip 5pt

 \emph{Step 2. The proof of $(ii)\Rightarrow (iii)$.}
\par
     Suppose that the claim $(ii)$ is true.
     We take $\widehat{\sigma}\in(0,\min\{1,(\max_{0\leq \ell\leq \hbar-1}\|e^{A^*(t_{\hbar}-t_{\ell})}\|_{\mathcal{L}(H)})^{-1}\})$ arbitrarily.
     By the claim $(ii)$, there are $\widehat{k}=\widehat{k}(\widehat \sigma)\in \mathbb{N}^+$ and $C(\widehat{\sigma})>0$
     so that (\ref{yu-3-8-1}) holds with $\sigma=\widehat{\sigma}$ and $C(\sigma)=C(\widehat{\sigma})$.
    We take $k^*\in\mathbb{N}^+$ so that $(k^*-1)\hbar\leq \widehat{k}< k^*\hbar$.
    Thus, by (\ref{yu-3-8-1}), we have
\begin{eqnarray*}\label{yu-3-9-1}
    \|\mathrm e^{A^*t_{k^*\hbar}}\varphi\|_H&\leq& C(\widehat{\sigma})
    \left(\sum_{j=1}^{\widehat{k}}\|B^*_{\nu(j)}\mathrm{e}^{A^*(t_{k^*\hbar}-t_j)}\varphi
    \|_U^2\right)^{\frac{1}{2}}+\widehat{\sigma}\|\mathrm{e}^{A^*(t_{k^*\hbar}
    -t_{\widehat{k}})}\varphi\|_H\nonumber\\
    &\leq& C(\widehat{\sigma})
    \left(\sum_{j=1}^{k^*\hbar-1}\|B^*_{\nu(j)}\mathrm{e}^{A^*(t_{k^*\hbar}-t_j)}\varphi
    \|_U^2\right)^{\frac{1}{2}}+ \widehat{\sigma}\max_{0\leq \ell\leq \hbar-1}\|\mathrm{e}^{A^*(t_{\hbar}-t_{\ell})}\|_{\mathcal{L}(H)}\|\varphi\|_H\nonumber\\
    &\leq& C(\widehat{\sigma})
    \left(\sum_{j=1}^{k^*\hbar-1}\|B^*_{\nu(j)}\mathrm{e}^{A^*(t_{k^*\hbar}-t_j)}\varphi
    \|_U^2\right)^{\frac{1}{2}}+ \sigma\|\varphi\|_H,
\end{eqnarray*}
    where $\sigma:=\widehat{\sigma}\max_{0\leq l\leq \hbar-1}\|\mathrm{e}^{A^*(t_{\hbar}-t_{\ell})}\|_{\mathcal{L}(H)}$. It is clear that
    $\sigma\in(0,1)$. Thus, the claim $(iii)$ is true with $k=k^*$ and $C=C(\widehat{\sigma})$.
\vskip 5pt
    \emph{Step 3. The proof of $(iii)\Rightarrow (iv)$.}
\par
    Assume that $(iii)$ holds. We take $\sigma\in (0,1)$, $k\in\mathbb{N}^+$ and $C>0$ so that (\ref{0315-1}) holds. Now, we prove $(iv)$ is true. We take
     $x_0\in H$ arbitrarily and, for each $\varepsilon>0$, introduce  a functional $\mathcal{J}_\varepsilon
     :H\to \mathbb{R}$ as follows:
\begin{equation*}\label{yu-5-6-5}
    \mathcal{J}_\varepsilon(\varphi):=\frac{1}{2}\sum_{j=1}^{k\hbar-1}\|B^*_{\nu(j)}
    e^{A^*(t_{k\hbar}-t_j)}\varphi\|_U^2+\langle \varphi, e^{A^*t_{k\hbar}}x_0\rangle_H
    +(\sigma\|x_0\|_H+\varepsilon)\|\varphi\|_H\textrm{ \,\,for each }\varphi\in H.
\end{equation*}

  First, we show that $\mathcal J_\varepsilon$ is coercive. For this purpose, given any sequence $\{\varphi_n\}_{n\in\mathbb{N}^+}\subset H\setminus\{0\}$ with $\|\varphi_n\|_H\to +\infty$ as $n\to+\infty$, we only need to prove that
\begin{equation}\label{yu-5-6-6}
    \liminf_{n\to+\infty}\frac{\mathcal{J}_\varepsilon(\varphi_n)}{\|\varphi_n\|_H}\geq \varepsilon.
\end{equation}
    Indeed,  let $\widehat{\varphi}_n:=
    \varphi_n/\|\varphi_n\|_H$. Then we have
\begin{equation}\label{yu-5-6-7}
    \frac{\mathcal{J}_\varepsilon(\varphi_n)}{\|\varphi_n\|_H}
    =\frac{1}{2}\|\varphi_n\|_H\sum_{j=1}^{k\hbar-1}\|B^*_{\nu(j)}
    e^{A^*(t_{k\hbar}-t_j)}\widehat{\varphi}_n\|_U^2+\langle \widehat{\varphi}_n,e^{A^*t_{k\hbar}}x_0\rangle_H+\sigma\|x_0\|_H+\varepsilon
\end{equation}
    for each $n\in\mathbb{N}^+$. There are only two cases may happen.
\par
    \emph{Case 1.} $\liminf_{n\to+\infty} \sum_{j=1}^{k\hbar-1}\|B^*_{\nu(j)}
    e^{A^*(t_{k\hbar}-t_j)}\widehat{\varphi}_n\|_U^2>0$. In this case, by (\ref{yu-5-6-7}), it is clear that
\begin{equation*}\label{yu-5-6-8}
    \liminf_{n\to+\infty}\frac{\mathcal{J}_\varepsilon(\varphi_n)}{\|\varphi_n\|_H}=+\infty.
\end{equation*}
    Thus, (\ref{yu-5-6-6}) is true.
\par
    \emph{Case 2.} $\liminf_{n\to+\infty} \sum_{j=1}^{k\hbar-1}\|B^*_{\nu(j)}
    e^{A^*(t_{k\hbar}-t_j)}\widehat{\varphi}_n\|_U^2=0$. Since $\|\widehat{\varphi}_n\|_H=1$ for each $n\in\mathbb{N}^+$, there is a subsequence of $\{\widehat{\varphi}_n\}_{n\in\mathbb{N}^+}$, still denoted  by the same way, and $\widehat{\varphi}$ so that
\begin{equation}\label{yu-5-6-b-8}
    \widehat{\varphi}_n\to \widehat{\varphi}\;\;\mbox{weakly in}\;\;H\;\;\mbox{as}\;\; n\to+\infty.
\end{equation}
     It is obvious that $\|\widehat{\varphi}\|_H\leq 1$. This, together with
    the assumption that  $\liminf_{n\to+\infty} \sum_{j=1}^{k\hbar-1}\|B^*_{\nu(j)}
    e^{A^*(t_{k\hbar}-t_j)}\widehat{\varphi}_n\|_U^2=0$, yields that
\begin{equation*}\label{yu-5-6-9}
    \sum_{j=1}^{k\hbar-1}\|B^*_{\nu(j)}
    e^{A^*(t_{k\hbar}-t_j)}\widehat{\varphi}\|_U^2=0.
\end{equation*}
    Hence, by (\ref{0315-1}), we have
\begin{equation*}\label{yu-5-6-10}
    0\leq (-\|e^{A^*t_{k\hbar}}\widehat{\varphi}\|_H+\sigma\|\widehat{\varphi}\|_H)\|x_0\|_H
    \leq \langle \widehat{\varphi}, e^{At_{k\hbar}}x_0\rangle_H+\sigma\|\widehat{\varphi}\|_H\|x_0\|_H.
\end{equation*}
    This, along with (\ref{yu-5-6-7}) and (\ref{yu-5-6-b-8}), gives that (\ref{yu-5-6-6}) holds in this case.

     In summary, we have proved that
    $\mathcal{J}_\varepsilon$ is coercive for any $\varepsilon>0$. It is obvious that $\mathcal J_\varepsilon$ is continuous and convex.
    Thus, for each $\varepsilon>0$, $\mathcal{J}_\varepsilon$ has a minimizer $\varphi^*_\varepsilon\in H$. Therefore,
   after some simple calculations, we have that   for any $\xi\in H$,
\begin{eqnarray}\label{yu-5-6-13}
    0\leq \lim_{\lambda\to0^+}\frac{\mathcal{J}_\varepsilon(\varphi^*_\varepsilon+\lambda\xi)
    -\mathcal{J}_\varepsilon(\varphi^*_\varepsilon)}{\lambda}
    &=&\langle e^{At_{k\hbar}}x_0,\xi\rangle_H+\sum_{j=1}^{k\hbar-1}\left\langle B^*_{\nu(j)}e^{A^*(t_{k\hbar}-t_j)}\varphi^*_\varepsilon, B^*_{\nu(j)}e^{A^*(t_{k\hbar}-t_j)}\xi\right\rangle_U\nonumber\\
    &\;&+
\begin{cases}
    (\sigma\|x_0\|_H+\varepsilon)
    \left\langle\frac{\varphi^*_\varepsilon}{\|\varphi^*_\varepsilon\|_H},\xi\right\rangle_H&\mbox{if}\;\;
    \varphi^*_\varepsilon\neq 0,\\
    (\sigma\|x_0\|_H+\varepsilon)\|\xi\|_H&\mbox{if}\;\;\varphi^*_\varepsilon=0.
\end{cases}
\end{eqnarray}

\par
    Let $u^*_\varepsilon:=(B^*e^{A^*(t_{k\hbar}-t_1)}\varphi^*_\varepsilon,
    B^*e^{A^*(t_{k\hbar}-t_2)}\varphi^*_\varepsilon,\cdots,
    B^*e^{A^*(t_{k\hbar}-t_{k\hbar-1})}\varphi^*_\varepsilon,0,\cdots)$. One can easily check that
\begin{equation}\label{yu-5-6-12}
    \langle x(t_{k\hbar};x_0,u^*_\varepsilon),\xi\rangle_H
    =\langle e^{At_{k\hbar}}x_0,\xi\rangle_H+\sum_{j=1}^{k\hbar-1}\left\langle B^*_{\nu(j)}e^{A^*(t_{k\hbar}-t_j)}\varphi^*_\varepsilon, B^*_{\nu(j)}e^{A^*(t_{k\hbar}-t_j)}\xi\right\rangle_U
\end{equation}
    for any $\xi\in H$.   This, together with (\ref{yu-5-6-13}), yields that
\begin{equation*}\label{yu-5-6-14}
    |\langle x(t_{k\hbar};x_0,u^*_\varepsilon),\xi\rangle_H|\leq (\sigma\|x_0\|_H+\varepsilon)\|\xi\|_H\;\;\mbox{for any}\;\;\xi\in H,
\end{equation*}
    which implies that
\begin{equation}\label{yu-5-6-15}
    \|x(t_{k\hbar};x_0,u^*_\varepsilon)\|_H\leq \sigma\|x_0\|_H+\varepsilon.
\end{equation}
    Moreover, by the definition of $u^*_\varepsilon$, it is obvious that
\begin{equation}\label{yu-5-6-16}
    \|u^*_\varepsilon\|_{l^2(\mathbb{N}^+;U)}^2=\sum_{j=1}^{k\hbar-1}
    \|B^*_{\nu(j)}e^{A^*(t_{k\hbar}-t_j)}\varphi^*_\varepsilon\|_U^2.
\end{equation}
    Thus, by (\ref{0315-1}), we have
\begin{equation*}\label{yu-5-8-1}
    -C\|u_\varepsilon^*\|_{l^2(\mathbb{N}^+;U)}\|x_0\|_H
    \leq (\sigma\|\varphi^*_\varepsilon\|_H-\|e^{A^*t_{k\hbar}}\varphi_\varepsilon^*\|_H)
    \|x_0\|_H
    \leq \sigma\|\varphi^*_\varepsilon\|_H\|x_0\|_H
    +\langle \varphi_\varepsilon^*,e^{At_{k\hbar}}x_0\rangle_H.
\end{equation*}
    This, together with (\ref{yu-5-6-16}) and the facts that  $\varepsilon>0$ and $\mathcal{J}_\varepsilon(\varphi^*_\varepsilon)\leq \mathcal{J}_\varepsilon(0)=0$, implies that
\begin{equation*}\label{yu-5-8-2}
    \|u^*_\varepsilon\|^2_{l^2(\mathbb{N}^+;U)}-2C\|u^*_\varepsilon\|_{l^2(\mathbb{N}^+;U)}
    \|x_0\|_H\leq 0.
\end{equation*}
    Hence, by the arbitrariness of $\varepsilon>0$, we get
\begin{equation}\label{yu-5-8-3}
    \|u^*_\varepsilon\|_{l^2(\mathbb{N}^+;U)}\leq 2C\|x_0\|_H\;\;\mbox{for any}\;\;\varepsilon>0.
\end{equation}
    It follows that, there is a sequence $\{\varepsilon_n\}_{n\in\mathbb{N}^+}\subset(0,+\infty)$ with $\varepsilon_n\to 0$ as $n\to+\infty$, and $u^*\in l^2(\mathbb{N}^+;U)$
    so that
\begin{equation}\label{yu-5-8-4}
    u^*_{\varepsilon_n}\to u^*\;\;\mbox{weakly in}\;\;l^2(\mathbb{N}^+;U)
    \;\;\mbox{as}\;\;n\to+\infty.
\end{equation}
    By (\ref{yu-5-8-3}), it is clear that
\begin{equation}\label{yu-5-8-5}
    \|u^*\|_{l^2(\mathbb{N}^+;U)}\leq 2C\|x_0\|_H.
\end{equation}
     Moreover, by (\ref{yu-5-8-4}), one can easily check that
\begin{equation*}\label{yu-5-8-6}
    x(t_{k\hbar};x_0, u^*_{\varepsilon_n})\to x(t_{k\hbar};x_0, u^*)
    \;\;\mbox{weakly in}\;\;H\;\;\mbox{as}\;\;n\to+\infty.
\end{equation*}
    This, along with (\ref{yu-5-6-15}), yields that
\begin{equation}\label{yu-5-8-7}
    \|x(t_{k\hbar};x_0, u^*)\|_H\leq \sigma\|x_0\|_H.
\end{equation}
    Therefore, by (\ref{yu-5-8-5}) and (\ref{yu-5-8-7}), we  conclude that (\ref{yu-3-9-bbb-2}) is true with $u=u^*$.  Thus $(iv)$ holds.
\vskip 5pt

    \emph{Step 4. The proof of $(iv)\Rightarrow (i)$.}
\par
    To this end, by Theorem \ref{thm-1}, we only need to show
\begin{equation}\label{yu-3-9-10}
    \mathcal{U}_{ad}(x_0)\neq \emptyset \;\;\mbox{for any}\;\;x_0\in H,
\end{equation}
    where $\mathcal{U}_{ad}(\cdot)$ is defined by (\ref{1015-4}). Let $\sigma\in (0,1)$, $k\in\mathbb{N}^+$ and $C>0$ be such that the claim $(iv)$ holds.
\par
    Let $x_0\in H$  be arbitrarily fixed.  Write $U^{k\hbar}:=\underbrace{U\times U\times\cdots\times U}_{k\hbar}$.
    It is clear that, for any $w=(w_1,\cdots,w_{k\hbar})\in U^{k\hbar}$,
\begin{equation*}\label{yu-3-9-14}
    x(t_j;x_0,w)= \mathrm{e}^{A(t_j-t_{j-1})}x(t_{j-1};x_0,w)+\mathrm{e}^{A(t_j-t_{j-1})}B_{j-1}w_{j-1}\;\;\mbox{for each}\;\;j\in\{1,2,\ldots,k\hbar\}.
\end{equation*}
   Here and in what follows  we agree that $B_0=0$ and $w_0=0$. It follows that
\begin{equation*}\label{yu-3-9-15}
    \|x(t_j;x_0,w)\|_H^2\leq C_A\|x(t_{j-1};x_0,w)\|_H^2+C_{A,\mathcal B_\hbar}\|w_{j-1}\|_U^2\;\;\mbox{for each}\;\;j\in\{1,2,\ldots,k\hbar\},
\end{equation*}
    where
    $$C_A:=\max\left\{1,2\max_{1\leq l\leq \hbar}\|\mathrm{e}^{A(t_l-t_{l-1})}\|_{\mathcal{L}(H)}^2\right\}
    \textrm{ and }
   C_{A,\mathcal B_\hbar}:=2\max_{1\leq l\leq \hbar}\left(\|\mathrm{e}^{A(t_l-t_{l-1})}\|_{\mathcal{L}(H)}^2
    \|B_\ell\|_{\mathcal{L}(U;H)}^2\right).
    $$
    Therefore, for each $j\in \{1,2,\ldots,k\hbar\}$,
\begin{eqnarray}\label{yu-3-9-16}
             \|x(t_j;x_0,w)\|_H^2
    &\leq& C_A^{j}\|x_0\|^2+C_{A,\mathcal B_\hbar}\left(\sum_{\ell=1}^{j-1}C_A^{j-1-\ell}\|w_\ell\|_U^2\right) \nonumber\\
    &\leq& C_A^{k\hbar}\|x_0\|^2+C_{A,\mathcal B_\hbar} C_A^{k\hbar-1}
    \left(\sum_{\ell=1}^{k\hbar-1}\|w_\ell\|_U^2
    \right).
\end{eqnarray}
    (Here, we used the fact $C_A\geq 1$.)

    By the claim $(iv)$, there is a control $u=(u_j)_{j\in\mathbb{N}}\in l^2(\mathbb{N}^+;U)$ so that (\ref{yu-3-9-bbb-2}) is true.
    Define $v^1:=(v_1^1,v_2^1,... ,v_{k\hbar}^1)\in U^{k\hbar}$ by
\begin{equation*}\label{yu-3-9-11}
    v^1_j=u_j\;\;\mbox{for each}\;\;j\in\{1,2,\ldots, k\hbar-1\}  \textrm{ and }v^1_{k\hbar}=0.
\end{equation*}
    This, together with  (\ref{yu-3-9-bbb-2}), yields that
\begin{equation}\label{yu-3-9-12}
    \|x(t_{k\hbar}^+;x_0,v^1)\|_H=\|x(t_{k\hbar};x_0,v^1)\|_H=\|x(t_{k\hbar};x_0,u)\|_H\leq \sigma\|x_0\|_H
\end{equation}
    and
\begin{equation}\label{yu-3-9-13}
    \|v^1\|_{U^{k\hbar}}\leq \|u\|_{l^2(\mathbb{N}^+;U)}\leq C
    \|x_0\|_H.
\end{equation}
    By (\ref{yu-3-9-13}) and (\ref{yu-3-9-16}), it is obvious that
\begin{equation*}\label{yu-3-9-17}
    \sum_{j=1}^{k\hbar}\|x(t_j;x_0,v^1)\|_H^2\leq C_{A,\mathcal B_\hbar,k}\|x_0\|^2_H,
\end{equation*}
    where $C_{A,\mathcal B_\hbar,k}:=k\hbar\left(C_A^{k\hbar}+C^2C_{A,\mathcal B_\hbar}C_A^{k\hbar-1}\right)$.
\par
    Replacing $x_0$ by $x(t^+_{k\hbar};x_0,v^1)$ in the  above arguments, we can show that
    there exists $v^2:=(v_1^2,v_2^2,\cdots,v_{k\hbar}^2)\in U^{k\hbar}$ so that
\begin{equation*}\label{yu-3-9-18}
    \|v^2\|_{U^{k\hbar}}\leq C\|x(t^+_{k\hbar};x_0,v^1)\|_H
\end{equation*}
    and
\begin{equation*}\label{yu-3-9-19}
    \sum_{j=1}^{k\hbar}\|x(t_j;x(t_{k\hbar}^+;x_0,v^1),v^2)\|^2_H
    \leq C_{A,\mathcal B_\hbar,k}\|x(t^+_{k\hbar};x_0,v^1)\|_H^2.
\end{equation*}
    These, along with (\ref{yu-3-9-12}), imply that
\begin{equation*}\label{yu-3-9-20}
    \|v^2\|_{U^{k\hbar}}\leq C\sigma\|x_0\|_H
\end{equation*}
    and
\begin{equation*}\label{yu-3-9-21}
    \sum_{j=1}^{k\hbar}\|x(t_j;x(t_{k\hbar}^+;x_0,v^1),v^2)\|_H
    \leq C_{A,\mathcal B_\hbar,k}\sigma^2\|x_0\|^2_H.
\end{equation*}
    By the same way, we can deduce that, for each $l\in\mathbb{N}^+$ with $l\geq 2$, there is
    $v^l:=(v_1^l,v_2^l,\cdots,v_{k\hbar}^l)\in U^{k\hbar}$ so that
\begin{equation}\label{yu-3-9-22}
    \|v^l\|_{U^{k\hbar}}\leq C\sigma^{l-1}\|x_0\|_H
\end{equation}
    and
\begin{equation}\label{yu-3-9-23}
    \sum_{j=1}^{k\hbar}\|x(t_j;x_{l-1}(t_{k\hbar}^+),v^l)\|^2_H\leq C_{A,\mathcal B_\hbar,k}\sigma^{2(l-1)}\|x_0\|^2,
\end{equation}
    where
\begin{equation}\label{yu-3-9-24}
\begin{cases}
    x_1(t_{k\hbar}^+):=x(t_{k\hbar}^+;x_0,v^1),\\
    x_l(t_{k\hbar}^+)=x(t_{k\hbar}^+;x_{l-1}(t_{k\hbar}^+),v^l)\;\;\mbox{for each}\;\;l\geq 2.
\end{cases}
\end{equation}
\par
    Define
\begin{equation*}\label{yu-3-9-25}
    v:=(v_1^1,v_2^1,\cdots,v_{k\hbar}^1,v_1^2,v_2^2,\cdots,v_{k\hbar}^2,\cdots).
\end{equation*}
    Since $\sigma\in(0,1)$, by (\ref{yu-3-9-13}) and (\ref{yu-3-9-22}), we obtain
\begin{equation}\label{yu-3-9-26}
    \|v\|^2_{l^2(\mathbb{N}^+;U)}\leq \sum_{l=1}^\infty \|v^l\|_{U^{k\hbar}}^2\leq
    C^2\left(\sum_{l=1}^{+\infty}\sigma^{2(l-1)}\right)\|x_0\|_H^2<+\infty.
\end{equation}
    Moreover,  by \eqref{yu-3-9-24},  for any $l\in \mathbb{N}^+$, one can easily check that
\begin{equation*}\label{yu-3-9-27}
    x(t_j;x_0,v)=x(t_{j};x_{l-1}(t_{k\hbar}^+),v^l)\;\;\mbox{for each}\;\;
    j\in\{(l-1)\hbar+1,\ldots,l\hbar\}.
\end{equation*}
    This, together with (\ref{yu-3-9-23}), yields that
\begin{equation}\label{yu-3-9-28}
    \sum_{j=1}^{+\infty}\|x(t_j;x_0,v)\|_H^2
    \leq \sum_{l=1}^{+\infty}\sum_{j=1}^{k\hbar}\|x(t_j;x_{l-1}(t_{k\hbar}^+),v^l)\|_H^2
    \leq C_{A,\mathcal B_\hbar,k}\left(\sum_{l=1}^{\infty}\sigma^{2(l-1)}\right)\|x_0\|_H^2<+\infty.
\end{equation}
         Thus, by (\ref{yu-3-9-26}) and (\ref{yu-3-9-28}), we  conclude that
    $v\in l^2(\mathbb{N}^+;U)$ and $(x(t_j;x_0,v))_{j\in\mathbb{N}^+}\in l^2(\mathbb{N}^+;H)$,
    i.e., $v\in\mathcal{U}_{ad}(x_0)$. This, along with the arbitrariness of $x_0$, implies that (\ref{yu-3-9-10}) is true. Hence, the claim $(i)$ is true.
\par
    In summary, we complete our proof.
\vskip 5mm

\section{Application to the coupled heat system with impulse controls}

In this section, we consider the system of heat equations coupled by constant matrices, which is a special case of \eqref{Intro-3}.
 Let
$\Omega\subset \mathbb{R}^N (N\in \mathbb{N}^+)$ be a bounded domain with a
smooth boundary $\partial \Omega$.
Let $\omega_k\subset \Omega$
 $(1\leq k\leq \hbar)$ be an open and nonempty subset of $\Omega$
 with $\omega:=\displaystyle{\cap_{k=1}^{\hbar}} \omega_k\not=\emptyset$.
 Denote by $\chi_{E}$ the characteristic function of the set $E\subset\mathbb R^N$.
Write $\triangle_n:=I_n\triangle=\mbox{diag}(\underbrace{\triangle,\triangle,\cdots,\triangle}_n)$, where
$I_n$ is the identity matrix in $\mathbb{R}^{n\times n}$ and
$\triangle$ is the Laplace operator with domain $D(\triangle):=H_0^1(\Omega)\cap H^2(\Omega)$.
In this section, we consider the system \eqref{Intro-3} with
\begin{equation*}\label{Intro-1}
    A:=\triangle_n+S,\;\;D(A):=(H_0^1(\Omega)\cap H^2(\Omega))^n \textrm{ and }B_k:=\chi_{\omega_k}D_k \textrm{ for each }1\leq k\leq \hbar,
\end{equation*}
 where  $S\in \mathbb{R}^{n\times n}$ and $\{D_k\}_{k=1}^\hbar\subset \mathbb R^{n\times m} \,(n,m\in\mathbb N^+)$.
  We let $\mathcal{D}:=(D_1,D_2,\cdots,D_{\hbar})$. It is obvious that if we let $H:=(L^2(\Omega))^n$
  and $U:=(L^2(\Omega))^m$, then $H=H^*$ and $U=U^*$. Moreover,
 $(A,D(A))$ generates a $C_0$-semigroup $\{\mathrm e^{At}\}_{t\in\mathbb{R}^+}$
over $(L^2(\Omega))^n$.
For simplicity,  we denote the system
 \begin{equation}\label{0314-1}
 \begin{cases}
x_t-\triangle_n x-Sx=0,& \textrm{ in } \Omega\times (\mathbb{R}^+\backslash \Lambda_\hbar),\\
x=0, &\textrm{ on }\partial \Omega\times \mathbb{R}^+,\\
x(t_j^+)=x(t_j)+\chi_{\omega_{\nu(j)}} D_{\nu(j)} u_j,&\textrm{ in } \Omega,\, j\in \mathbb N^+
\end{cases}
\end{equation}
by  $[S,\{\chi_{\omega_k}D_k\}_{k=1}^{\hbar},\Lambda_\hbar]$  instead of $[A,\mathcal B_\hbar, \Lambda_\hbar]$ in this section,
where  $\Lambda_\hbar:=\{t_j\}_{j\in\mathbb N}\in \mathcal J_\hbar$.
We will prove the following stabilization  results for the system \eqref{0314-1}, which is an application of Theorem \ref{yu-main-theorem}.
\begin{theorem}\label{thm-2}
Given $S\in\mathbb R^{n\times n}$ and $\{D_k\}_{k=1}^\hbar\subset \mathbb R^{n\times m}$. The  following statements are equivalent.
\begin{itemize}
        \item[(i)]
                There exists $\Lambda_\hbar\in \mathcal J_\hbar$ so that  the system $[S,\{\chi_{\omega_k}D_k\}_{k=1}^\hbar,\Lambda_\hbar]$ is exponentially $\hbar$-stabilizable.
        \item[(ii)]
            $
                 \textrm{rank}\,(\lambda I_n-S, \mathcal{D})=n \textrm{ for any } \lambda\in\mathbb C \textrm{ with } \textrm{Re}(\lambda)\geq \lambda_1.
         $
       \item[(iii)]
                   $
                 \textrm{rank}\,(\lambda I_n-S, \mathcal{D})=n \textrm{ for any } \lambda\in\mathbb \sigma(S) \textrm{ with } \textrm{Re}(\lambda)\geq\lambda_1.
           $
\end{itemize}
  Here and throughout this section, $\lambda_1$ denotes the first eigenvalue of $-\triangle$ and $\sigma(S)$  the spectrum of the matrix $S$.
\end{theorem}
In order to prove Theorem \ref{thm-2}, we first introduce some definitions and two lemmas.
For any $E\in \mathbb R^{k\times k}$, $F\in \mathbb R^{k\times \ell}$ ($k, \ell\in\mathbb N^+$), we define (see \cite{Q-W} and  \cite{Q-W-Y})
\begin{equation}\label{0314-3}
d_E:=\min\left\{\frac{\pi}{\textrm{Im} (\lambda)}:\, \lambda\in \sigma(E)\right\} \textrm{ and }q(E,F):= \max\{ \textrm{dim} \mathcal V_E(f):\,f \textrm{ is a column of }F\},
\end{equation}
where $\mbox{Im}(\lambda)=\lambda_2$ if $\lambda=\lambda_1+i\lambda_2$, $\mathcal V_E(f) := \textrm{ span}\{f,Ef,\cdots, E^{k-1}f\}$ and we agree that $\frac{1}{0}=+\infty$.
Moreover, we define
\begin{equation*}
\begin{split}
\mathcal L_{E,F,\hbar} :=
\big\{
       \{\tau_j\}_{j\in\mathbb N}: \; 0=\tau_0<\tau_1<\cdots<\tau_n<\cdots,\qquad\qquad\qquad\qquad\qquad\qquad\qquad\quad\\
                 \,\textrm{Card}((s,s+d_E)\cap \{\tau_j\}_{j\in\mathbb N})\geq \hbar q(E,F)+2\textrm{ for any }s\in\mathbb R^+
 \big\}.
 \end{split}
\end{equation*}

The following result can be found in the proof of Theorem 2.2 in \cite{Q-W}.
\begin{lemma}\label{theo-1} Let $E\in \mathbb R^{k\times k}$, $F\in \mathbb R^{k\times \ell}$ ($k, \ell\in\mathbb N^+$).
 For each strictly increasing sequence
$\{\tau_i\}_{i=1}^{q(E,F)}$ with
$\tau_{q(E,F)}-\tau_1<d_{E}$, we have that
\begin{equation*}
    \mbox{span}\;\{e^{-E\tau_1}F, e^{{-E}\tau_2}F,
    \cdots, e^{-E\tau_{q(E,F)}}F\}
    =\mbox{span}\;\{F, EF, \cdots, {E}^{k-1}F\}.
\end{equation*}
\end{lemma}

In the following of this section, we discuss for  fixed $\hbar\in\mathbb N^+,\,S\in \mathbb R^{n\times n}$ and $\{D_k\}_{k=1}^\hbar\subset \mathbb R^{n\times m}$.
Write $\mathcal{D}:=(D_1,\cdots, D_\hbar)\in\mathbb R^{n\times m\hbar}$.
Lemma \ref{Preli-2} is quoted from \cite{wang-yan-yu}.

\begin{lemma}\label{Preli-2}
    Let $\Lambda_{\hbar}=\{t_j\}_{j\in\mathbb N}\in\mathcal{J}_{\hbar}$ and $\widehat k\in\mathbb{N}^+$ be fixed. The following two claims are equivalent:
\begin{enumerate}
  \item [(i)] $\text{rank}\,\left(\mathrm e^{-S t_1}D_{\nu(1)},\mathrm e^{-S t_2}D_{\nu(2)},\cdots, \mathrm e^{-S t_{\widehat k}}D_{\nu(\widehat k)}\right)=n$.
  \item [(ii)] There are two constants $\theta\in(0,1)$ (independent of $\widehat k$) and $C(\widehat k)>0$ so that
\begin{eqnarray}\label{0317-1}
 &\;&\|\mathrm e^{A^*t_{\widehat k+1}}\varphi\|_{(L^2(\Omega))^n}\nonumber\\
 &\leq& C(\widehat k)\left(\displaystyle{\sum_{j=1}^{\widehat k}}
    \|B^*_{\nu(j)}\mathrm e^{A^*(t_{\widehat k+1}-t_j)}\varphi\|_{(L^2(\Omega))^m}\right)^{\theta}
    \|\varphi\|^{1-\theta}_{(L^2(\Omega))^n}\;\;\mbox{for all}\;\;\varphi\in (L^2(\Omega))^n.
\end{eqnarray}
\end{enumerate}
\end{lemma}

Based on Lemma \ref{theo-1} and Lemma \ref{Preli-2},  we obtain the following Lemma.
\begin{lemma}\label{thm-4}
Let $\Lambda_\hbar=\{t_j\}_{j\in\mathbb N}\in\mathcal J_\hbar\cap \mathcal L_{S,\mathcal{D},\hbar}$.  Suppose that
$
\text{rank}\,\left(\mathcal{D},S\mathcal{D},\cdots, S^{n-1}\mathcal{D}\right)=n.
$
Then
the system $[S,\{\chi_{\omega_k}D_k\}_{k=1}^{\hbar},\Lambda_\hbar]$ is exponentially $\hbar$-stabilizable.

 \end{lemma}

\begin{proof}
Since $\Lambda_\hbar\in \mathcal L_{S,\mathcal{D},\hbar}$, by the definition of $q(S,\mathcal{D})$ (see \eqref{0314-3}) and the fact that $\mathcal{D}=(D_1,\cdots, D_\hbar)$, it is easy to see
\begin{equation*}
t_{j+q(S,D_j)\hbar}-t_j\leq t_{j+q(S,\mathcal{D})\hbar}-t_j<d_S\,\textrm{ for each }j\in\{1,\ldots,\hbar\}.
\end{equation*}
This, along with Lemma \ref{theo-1}, shows that
\begin{equation*}
\textrm{span}\,\{\mathrm{e}^{-St_j}D_j,\cdots, \mathrm{e}^{-St_{j+q(H,D_j)\hbar}}D_j\}=\textrm{span}\,\{D_j,SD_j,\cdots, S^{n-1}D_j\} \textrm{ for each } j\in \{1,\ldots,\hbar\}.
\end{equation*}
From this, we obtain that
\begin{eqnarray}\label{0321-1}
  \sum_{j=1}^\hbar \textrm{span}\,\{\mathrm e^{-St_j}D_j,\cdots, \mathrm e^{-St_{j+q(S,D_j)\hbar}}D_j\}
&=&\sum_{j=1}^\hbar \textrm{span}\,\{D_j,SD_j,\cdots,S^{n-1}D_j\}\nonumber\\
&=&\textrm{span}\;\{\mathcal{D},S\mathcal{D},\cdots,S^{n-1}\mathcal{D}\}.
\end{eqnarray}
    Here, for two linear spaces $V_1$ and $V_2$, we let $V_1+V_2=\mbox{span}\{V_1,V_2\}$. Noting that for each $j\in\{1,\ldots,\hbar\}$, $\nu(j)=\nu(j+\ell\hbar)$ for any $\ell\in\mathbb N$, \eqref{0321-1} implies that
\begin{equation}\label{0321-2}
\textrm{span}\,\{\mathrm{e}^{-St_1}D_{\nu(1)},\cdots, \mathrm{e}^{-St_{\hbar+q(S,\mathcal{D})\hbar}}D_{\nu\left(\hbar
+q(S,\mathcal{D})\hbar\right)}\}= \textrm{span}\,\{\mathcal{D},S\mathcal{D},\cdots,S^{n-1}\mathcal{D}\}.
\end{equation}
Denote $\widehat n=\hbar+q(S,\mathcal{D})\hbar$. By (\ref{0321-2}) and our assumption, we have that
$
\textrm{rank} \,(\mathrm{e}^{-St_1}D_{\nu(1)},\cdots, \mathrm{e}^{-St_{\widehat n}}D_{\nu(\widehat n)})=n.
$
This, along with Lemma \ref{Preli-2}, shows that there exists $\theta\in(0,1)$ and $C(\widehat n)>0$
so that \eqref{0317-1} stands with $\widehat k$ replaced by $\widehat n$.
Then for each $\alpha>0$, there exists  a constant $C(\widehat n,\alpha)>0$ so that
\begin{equation*}
 \|\mathrm e^{A^*t_{\widehat n+1}}\varphi\|_{(L^2(\Omega))^n}
 \leq C(\widehat n,\alpha)\left(\displaystyle{\sum_{j=1}^{\widehat n}}
    \|B^*_{\nu(j)}\mathrm e^{A^*(t_{\widehat n+1}-t_j)}\varphi\|^2_{(L^2(\Omega))^m}\right)^{\frac{1}{2}}
   +\alpha \|\varphi\|_{(L^2(\Omega))^n}\;\;\mbox{for all}\;\;\varphi\in (L^2(\Omega))^n.
\end{equation*}
This, along with Theorem \ref{yu-main-theorem}, completes the proof.
%\begin{equation}
%\begin{array}{lll}
% &&\|\mathrm e^{A^*t_{\widehat n+1}}\varphi\|_{(L^2(\Omega))^n}\\
% &\leq& M(\widehat n)\left(\displaystyle{\sum_{j=1}^{\widehat n}}
%    \|B^*_{\nu(j)}\mathrm e^{A^*(t_{\widehat n+1}-t_j)}\varphi\|_{(L^2(\Omega))^m}\right)^{\theta}
%    \|\varphi\|^{1-\theta}_{(L^2(\Omega))^n}\;\;\mbox{for all}\;\;\varphi\in (L^2(\Omega))^n.
%\end{array}
%\end{equation}
%Therefore, for any $\alpha>0$, there exists $M=M(\widehat n, \alpha)>0$ so that \eqref{0317-1} stands.
\end{proof}

\begin{lemma}\label{lem-2}(Kalman controllability decomposition (see \cite{Sontag})) The  following statements are equivalent:
\begin{enumerate}
\item[(i)] rank$\,(\mathcal{D},S\mathcal{D},\cdots, S^{n-1}\mathcal{D})=n_1<n$;

\item[(ii)] There exists an invertible matrix $J\in\mathbb R^{n\times n}$ so that
\begin{equation}\label{0322-1}
J^{-1}SJ=\left(
\begin{array}{cc}
S_1 & S_2 \\
0     & S_3
\end{array}
\right) \textrm{ and }
J^{-1}\mathcal{D}=\left(
\begin{array}{c}
\widetilde{\mathcal{D}} \\
0     \\
\end{array}
\right),
\end{equation}
where $S_1\in \mathbb R^{n_1\times n_1}$, $\widetilde{\mathcal{D}}\in \mathbb R^{n_1\times m\hbar}$  and
$
\textrm{rank}\,(\widetilde{\mathcal{D}},S_1\widetilde{\mathcal{D}},\cdots, S_1^{n_1-1}\widetilde{\mathcal{D}})=n_1.
$
\end{enumerate}

\end{lemma}

\begin{proposition}\label{thm-5}
 Let  $\Lambda_\hbar=\{t_j\}_{j\in\mathbb{N}}\in \mathcal J_\hbar\cap \mathcal L_{S,\mathcal{D},\hbar}$. Suppose that
\begin{equation}\label{1201-1}
\textrm{ rank}(\lambda I -S,\mathcal{D})=n \textrm{ for any }\lambda\in\mathbb C \textrm{ with }\textrm{Re}(\lambda)\geq \lambda_1.
 \end{equation}
Then
the system $[S,\{\chi_{\omega_k}D_k\}_{k=1}^{\hbar},\Lambda_\hbar]$ is exponentially $\hbar$-stabilizable.
\end{proposition}

\begin{proof}
Without loss of generality, we assume that $\textrm{rank}\,(\mathcal{D},S\mathcal{D}, \cdots, S^{n-1}\mathcal{D})=n_1$ with $1\leq n_1<n$,
since when $n_1=n$ the conclusion follows from Lemma \ref{thm-4}  at once.
By Lemma \ref{lem-2}, there exists  an invertible matrix $J\in \mathbb R^{n\times n}$ so that
   (\ref{0322-1}) holds with  $S_1\in \mathbb R^{n_1\times n_1}$, $\widetilde{\mathcal{D}}\in \mathbb R^{n_1\times m\hbar}$  and
\begin{equation}\label{0317-7}
\textrm{rank}\,(\widetilde{\mathcal{D}},S_1\widetilde{\mathcal{D}},\cdots, S_1^{n_1-1}\widetilde{\mathcal{D}})=n_1.
\end{equation}
Denote
$
J^{-1}D_1=\left(
\begin{array}{c}
\widetilde D_1 \\
0     \\
\end{array}
\right),\textrm{...},
J^{-1}D_\hbar=\left(
\begin{array}{c}
\widetilde D_\hbar \\
0     \\
\end{array}
\right),
$
where $\widetilde D_j\in \mathbb R^{n_1\times m}$. Then
\begin{equation}\label{0317-6}
\widetilde{\mathcal{D}}=(\widetilde D_1,\cdots, \widetilde D_\hbar).
\end{equation}

First, we show that
\begin{equation}\label{1124-3}
\textrm{Re}(\lambda)<\lambda_1 \textrm{ for any }\lambda\in\sigma (S_3).
\end{equation}
For otherwise, there would exist  $\lambda_0\in\sigma (S_3)$ so that Re$(\lambda_0)\geq \lambda_1$.
It follows that
\begin{equation*}
\textrm{rank}(\lambda_0I-S,\mathcal{D})=\textrm{rank}(\lambda_0I-J^{-1}SJ, J^{-1}\mathcal{D}) =\textrm{rank}\left(
\begin{array}{ccc}
\lambda_0I-S_1 & S_2 & \widetilde{\mathcal{D}}  \\
0     & \lambda_0I-S_3 & 0
\end{array}
\right)<n,
\end{equation*}
which contradicts \eqref{1201-1}. Thus  \eqref{1124-3} stands.
Therefore, there exist constants $M_0>0$ and $\varepsilon_0>0$ so that the solution   $\widetilde{y}(\cdot)$ of system
\begin{equation}\label{1124-4}
\begin{cases}
\widetilde{y}_t-\triangle_{n-n_1} \widetilde{y}-S_3^\top\widetilde{y}=0, & \textrm{ in } \Omega\times \mathbb{R}^+,\\
\widetilde{y}=0, &\textrm{ on }\partial \Omega\times \mathbb{R}^+
\end{cases}
\end{equation}
satisfies that
\begin{equation}\label{1201-7}
\|\widetilde{y}(t)\|_{(L^2(\Omega))^{n-n_1}}\leq M_0 \mathrm e^{-\varepsilon_0 t}\|\widetilde{y}(0)\|_{(L^2(\Omega))^{n-n_1}} \textrm{ for any } t>0.
\end{equation}

Second, by the definitions of $d_S$ and $q(S,\mathcal{D})$ (see \eqref{0314-3}), we obtain that
$
d_S\leq d_{S_1}$
and
$
q(S_1,\widetilde{\mathcal{D}})\leq q(S,\mathcal{D}).
$
This, along with the fact that  $\Lambda_\hbar\in \mathcal J_\hbar\cap \mathcal L_{S,\mathcal{D},\hbar}$, shows that
$\Lambda_\hbar\in \mathcal J_\hbar\cap \mathcal L_{S_1, \widetilde{\mathcal{D}},\hbar}$. Then, by
\eqref{0317-7},  \eqref{0317-6}  and Lemma \ref{thm-4}, we have that the system
$[S_1,\{ \chi_{\omega_k} \widetilde D_k \}_{k=1}^\hbar, \Lambda_\hbar]$ is exponentially $\hbar$-stabilizable.
Thus there exists a feedback control law $\widehat{\mathcal F}=\{\widehat{F}_j\}_{j=1}^\hbar\subset \mathcal L ((L^2(\Omega))^{n_1};(L^2(\Omega))^m)$ so that the system
 \begin{equation}\label{1201-5}
\begin{cases}
z_t-\triangle_{n_1} z-S_1z=0, &\textrm{ in } \Omega\times (\mathbb{R}^+\backslash \Lambda_\hbar),\\
z=0, & \textrm{ on } \partial \Omega\times \mathbb{R}^+,\\
z(t_j^+)=z(t_j)+\chi_{\omega_{\nu(j)}}\widetilde D_{\nu(j)} \widehat{F}_{\nu(j)} z_1(t_j), &\textrm{ in } \Omega, j\in\mathbb N^+
\end{cases}
\end{equation}
is stable, i.e., if we  denote the solution of the system \eqref{1201-5} by $z(t)=S_{\widehat{\mathcal F}}(t) z(0)$, then there
exist constants $M_1>0$ and $\varepsilon_1\in(0,\varepsilon_0)$, so that
\begin{equation*}
\|z(t)\|_{(L^2(\Omega))^{n_1}}
=\|S_{\widehat{\mathcal F}}(t)z(0)\|_{(L^2(\Omega))^{n_1}}\leq M_1 \mathrm e^{-\varepsilon_1 t}\|z(0)\|_{(L^2(\Omega))^{n_1}} \textrm{ for any } t\geq 0.
\end{equation*}
This indicates that
\begin{equation}\label{1201-6}
\|S_{\widehat{\mathcal F}}(t)\|_{\mathcal L((L^2(\Omega))^{n_1})}\leq M_1 \mathrm e^{-\varepsilon_1 t} \textrm{ for any }t\geq 0.
\end{equation}

Third, we consider the system
\begin{equation}\label{1124-7}
\begin{cases}
\widehat{y}_t-\triangle_{n_1} \widehat{y}-S_1\widehat{y}=S_2\widetilde{y}, & \textrm{ in } \Omega\times (\mathbb{R}^+\backslash \Lambda_\hbar),\\
\widehat{y}=0, &\textrm{ on } \partial \Omega\times \mathbb{R}^+,\\
\widehat{y}(t_j^+)=\widehat{y}(t_j)+\chi_{\omega_{\nu(j)}}\widetilde D_{\nu(j)} \widehat{F}_{\nu(j)}  \widehat{y}(t_j), &\textrm{ in } \Omega, j\in\mathbb N^+,
\end{cases}
\end{equation}
    where $\widetilde{y}(\cdot)$ is a solution of (\ref{1124-4}).
One can easily verify that the solution to the system \eqref{1124-7} satisfies that
\begin{equation*}
\widehat{y}(t)=S_{\widehat{\mathcal F}}(t) \widehat{y}(0)+\int_0^t S_{\widehat{\mathcal F}}(t-s) S_2 \widetilde{y}(s)\,\mathrm ds \textrm{ for any }t\geq0.
\end{equation*}
This, combined with \eqref{1201-6} and \eqref{1201-7} , shows that
\begin{equation}\label{0314-5}
\|\widehat{y}(t)\|_{(L^2(\Omega))^{n_1}}
               \leq  M_2\mathrm e^{-\varepsilon_1 t}\left(\|\widehat{y}(0)\|_{(L^2(\Omega))^{n_1}}
               +\|\widetilde{y}(0)\|_{(L^2(\Omega))^{n-n_1}}\right) \textrm{ for any } t\geq 0,
\end{equation}
where  $M_2:=\max\{M_1,M_1M_0\|S_2\|_{\mathcal{L}(\mathbb R^{n-n_1};\mathbb{R}^{n_1})}\}$.
%\begin{eqnarray}\label{0314-5}
%\|y_1(t)\|_{(L^2(\Omega))^{n_1}} & \leq &\|S_\mathcal F(t)y_1(0)\|_{(L^2(\Omega))^{n_1}}
%               +\int_0^t \|S_\mathcal F(t-s)\|_{\mathcal L((L^2(\Omega))^{n_1})} \|C_2 y_2(s)\|_{(L^2(\Omega))^{n_1}} \,\mathrm ds\\
%              & \leq & M_1 \mathrm e^{-\varepsilon_1 t}\|y_1(0)\|_{(L^2(\Omega))^{n_1}}+M_1 M_2 \|C_2\|_{\mathbb R^{n_1\times (n-n_1)}}\frac{1}{\varepsilon_2-\varepsilon_1}
%               \mathrm e^{-\varepsilon_1t}\|y_2(0)\|_{(L^2(\Omega))^{n-n_1}}\nonumber\\ %\frac{1}{\varepsilon_1-\varepsilon_2} \mathrm e^{(\varepsilon_1-\varepsilon)t}\\
%              & \leq & M' \mathrm e^{-\varepsilon_1 t}\left(\|y_1(0)\|_{(L^2(\Omega))^{n_1}}+\|y_2(0)\|_{(L^2(\Omega))^{n-n_1}}\right) \textrm{ for any } t\geq 0,\nonumber
%\end{eqnarray}
%where  $M'=\max\{M_1,M_1M_2\|C_2\|_{\mathbb R^{n_1\times (n-n_1)}}\}$.

Finally, let
$
x(t)=J
\left( \begin{array}{c}
\widehat{y}\\
\widetilde{y}
\end{array}\right)(t) \textrm{ for any } t\geq 0
\textrm{ and } F_j=\left(
 \begin{array}{cc}
 \widehat{F}_j & 0
 \end{array}
 \right) J^{-1} \textrm{ for each } j\in\{1,\ldots,\hbar\}.
$
 Then by \eqref{0322-1}, \eqref{1124-4} and \eqref{1124-7}, we have $x(\cdot)$ satisfies
\begin{equation*}\label{1124-8}
\begin{cases}
x_t-\triangle_n x-Sx=0,& \textrm{ in } \Omega\times (\mathbb{R}^+\backslash \Lambda_\hbar),\\
x=0, &\textrm{ on }\partial  \Omega\times \mathbb{R}^+\\
x(t_j^+)=x(t_j)+\chi_{\omega_{\nu(j)}} D_{\nu(j)} F_{\nu(j)} x(t_j),&\textrm{ in } \Omega, j\geq 1.
\end{cases}
\end{equation*}
Moreover,  from \eqref{1201-7} and \eqref{0314-5} and the fact that $0<\varepsilon_1<\varepsilon_0$, we obtain that
\begin{equation*}
\|x(t)\|_{(L^2(\Omega))^n}\leq \|J\|_{\mathcal{L}(\mathbb{R}^n;\mathbb R^{n})}\cdot\|y(t)\|_{(L^2(\Omega))^n}\leq M_3 \mathrm e^{-\varepsilon_1 t}\|x (0)\|_{(L^2(\Omega))^n} \textrm{ for any } t\geq 0,
\end{equation*}
where $M_3:=2(M_0+M_2)\|J\|_{\mathcal{L}(\mathbb{R}^n;\mathbb R^{n})}\|J^{-1}\|_{\mathcal{L}(\mathbb{R}^n;\mathbb R^{n})}$.   This completes the proof.
\end{proof}

\begin{proof}[Proof of Theorem \ref{thm-2}]
It is easy to see that $(ii)$ and $(iii)$ are equivalent. It follows from Proposition \ref{thm-5}  that $(ii)\Rightarrow(i)$.

Next, we prove $(i)\Rightarrow(iii)$ by contradiction.
Suppose there exists
a complex number $\lambda_0\in \sigma(S)$ with Re$(\lambda_0)\geq \lambda_1$ so that
\begin{equation}\label{0116-4}
\textrm{rank}\,(\lambda_0 I_n-S, \mathcal{D}):=n_1<n.
\end{equation}
Thus by Lemma \ref{lem-2}, there exists an invertible matrix $J\in\mathbb R^{n\times n}$ so that
\begin{equation}\label{0116-5}
J^{-1}SJ=\left(
\begin{array}{cc}
S_1 & S_2 \\
0     & S_3
\end{array}
\right) \textrm{ and }
J^{-1}D_1=\left(
\begin{array}{c}
\widetilde D_1 \\
0     \\
\end{array}
\right),\cdots,
J^{-1}D_\hbar=\left(
\begin{array}{c}
\widetilde D_\hbar \\
0     \\
\end{array}
\right),
\end{equation}
where $S_1\in \mathbb R^{n_1\times n_1}$, $\widetilde D_k\in \mathbb R^{n_1\times m}$ for each $k=1,2,\ldots,\hbar$,  and by
writing $\widetilde{\mathcal{D}}=(\widetilde D_1,\widetilde D_2,\cdots, \widetilde D_\hbar)$,
\begin{equation}\label{0116-6}
\textrm{rank}\,(\widetilde{\mathcal{D}},S_1 \widetilde{\mathcal{D}},\cdots, S_1^{n_1-1}\widetilde{\mathcal{D}})=n_1.
\end{equation}
By (\ref{0116-4}) and (\ref{0116-5}), we have that
\begin{equation*}
\begin{split}
   \textrm{rank}(\lambda_0I_n-S,\mathcal{D})
=\textrm{rank}(\lambda_0I_n-J^{-1}SJ, J^{-1}\mathcal{D})
=\textrm{rank}\left(
\begin{array}{ccc}
\lambda_0I_{n_1}-S_1 & S_2 & \widetilde{\mathcal{D}}  \\
0     & \lambda_0I_{n-n_1}-S_3 & 0
\end{array}
\right)=n_1,
\end{split}
\end{equation*}
This, along with \eqref{0116-6}, shows that $\lambda_0$ is an eigenvalue of $S_3$.

    Since $(i)$ is true, let  $\Lambda_\hbar=\{t_j\}_{j\in\mathbb N}\in\mathcal J_\hbar$ be such that
the system $[S,\{\chi_{\omega_k}D_k\}_{k=1}^\hbar$,$\Lambda_\hbar]$ is exponentially  $\hbar$-stabilizable.
Then  there exists a feedback control law $\mathcal F=\{F_k\}_{k=1}^\hbar$,
  two positive constants $C$ and $\alpha$ so that the solution $x_{\mathcal F}(\cdot)$ to  the  closed-loop system
 \begin{equation}\label{0116-10}
 \begin{cases}
x_t-\triangle_n x-Sx=0,& \textrm{ in } \Omega\times (\mathbb{R}^+\backslash \Lambda_\hbar),\\
x=0, &\textrm{ on }\partial \Omega\times \mathbb{R}^+,\\
x(t_j^+)=x(t_j)+\chi_{\omega_{\nu(j)}} D_{\nu(j)} F_{\nu(j)}x(t_j),&\textrm{ in } \Omega, j\in \mathbb N^+,
\end{cases}
\end{equation}
satisfies that
\begin{equation}\label{0116-8}
\|x_{\mathcal F}(t)\|_{(L^2(\Omega))^n}\leq C \mathrm e^{-\alpha t} \|x_{\mathcal{F}}(0)\|_{(L^2(\Omega))^n} \textrm{ for any } t\geq 0.
\end{equation}
Let
$y(\cdot)=J^{-1}x_{\mathcal F}(\cdot)
=\left(
   \begin{array}{l}
   \widehat{y}(\cdot)\\
   \widetilde{y}(\cdot)
   \end{array}
   \right)$,
    where $\widehat{y}(t)\in (L^2(\Omega))^{n_1}$ and $\widetilde{y}(t)\in (L^2(\Omega))^{n-n_1}$ for each $t\geq 0$.
   Then by \eqref{0116-8}, there exists $C'>0$ so that
   \begin{equation}\label{0116-9}
   \|\widetilde{y}(t)\|_{(L^2(\Omega))^{n-n_1}}\leq  C' \mathrm e^{-\alpha t} \|y(0)\|_{(L^2(\Omega))^{n}}  \textrm{ for any } t\geq 0,
   \end{equation}
    and by (\ref{0116-5}) and \eqref{0116-10}, $\widetilde{y}(\cdot)$ verifies that
\begin{equation*}\label{0116-2}
 \begin{cases}
\widetilde{y}_t-\triangle_{n-n_1} \widetilde{y}-S_3\widetilde{y}=0,& \textrm{ in } \Omega\times \mathbb{R}^+,\\
\widetilde{y}=0, &\textrm{ on }\partial \Omega\times \mathbb{R}^+,\\
\widetilde{y}(0)=LJ^{-1}x_{\mathcal{F}}(0), &\textrm{ in }\Omega,
\end{cases}
\end{equation*}
    where $L:=\mbox{diag}(0\cdot I_{n_1},I_{n-n_1})$.
   Let $\xi\in\mathbb C^{n-n_1}$ be an eigenvector of $S_3$ with respect to $\lambda_0$ and $e_1$ be the
unitary eigenvector of $-\triangle$ with  respect to $\lambda_1$. Let
$
    x_{\mathcal{F}}(0)=J\left(
                          \begin{array}{c}
                            0 \\
                            \xi e_1 \\
                          \end{array}
                        \right).
$
    One can easily check that $\widetilde{y}(0)=\xi e_1$ and $\widetilde{y}(t)=\mathrm e^{(\lambda_0-\lambda_1)t}\xi e_1$ for any $t\geq 0$.
    It contradicts \eqref{0116-9}
since $\textrm{Re}(\lambda_0)\geq \lambda_1$. Thus, the claim $(iii)$ holds.
This completes the proof of $(i)\Rightarrow(iii)$.

In summary, we complete the proof of Theorem \ref{thm-2}.
\end{proof}

\begin{remark}
In this section, we give a necessary and sufficient condition for the exponentially $\hbar$-stabilization of the coupled heat equations (see Theorem \ref{thm-2}).
Moreover,  Proposition \ref{thm-5} tells  how to choose the impulse instants so that the system can be  exponentially $\hbar$-stabilized.
\end{remark}

\end{document}